\documentclass[10pt]{article}
\usepackage{latexsym}
\usepackage{indentfirst}
\usepackage{amssymb,amsfonts,amsmath,amsthm}
\usepackage{graphicx}
\usepackage{mathrsfs}
\usepackage{array}
\usepackage{color}
\usepackage{epstopdf}
\usepackage[all]{xy}
\usepackage{hyperref}
\usepackage{url}
\usepackage{tikz-cd}
\usepackage{mathtools}
\usepackage{booktabs}

\usepackage{enumitem}

\setlist[enumerate]{
  leftmargin=*,
  rightmargin=0pt,
  labelindent=0pt,
  labelsep=0.5em,
  align=left
}

\usepackage{geometry}
\newtheorem{mainthm}{Theorem}

\newtheorem{thm}{Theorem}[section]
\newtheorem{prop}[thm]{Proposition}
\newtheorem{lem}[thm]{Lemma}

\newtheorem{cor}[thm]{Corollary}

\newtheorem{rmk}[thm]{Remark}
\newtheorem{dfn}[thm]{Definition}

\numberwithin{equation}{section}

\newcommand{\frakm}{{\mathfrak m}}

\newcommand{\frakp}{{\mathfrak p}}

\newcommand{\frakL}{{\mathfrak L}}
\newcommand{\frakM}{{\mathfrak M}}

\newcommand{\frakS}{{\mathfrak S}}

\newcommand{\bA}{{\mathbb A}}

\newcommand{\bC}{{\mathbb C}}

\newcommand{\bF}{{\mathbb F}}

\newcommand{\bL}{{\mathbb L}}
\newcommand{\bM}{{\mathbb M}}
\newcommand{\bN}{{\mathbb N}}

\newcommand{\bP}{{\mathbb P}}
\newcommand{\bQ}{{\mathbb Q}}

\newcommand{\bZ}{{\mathbb Z}}

\newcommand{\calA}{{\mathcal A}}

\newcommand{\calN}{{\mathcal N}}
\newcommand{\calO}{{\mathcal O}}

\newcommand{\calU}{{\mathcal U}}

\newcommand{\Ainf}{{\mathrm{A_{inf}}}}

\newcommand{\Bcris}{{\mathrm{B_{cris}}}}

\newcommand{\HT}{{\mathrm{HT}}}

\newcommand{\gr}{{\mathrm{gr}}}
\newcommand{\PGL}{{\mathrm{PGL}}}

\newcommand{\Coker}{{\mathrm{Coker}}}       %cokernal
\newcommand{\End}{{\mathrm{End}}}           %End-functor
\newcommand{\Ext}{{\mathrm{Ext}}}           %Ext-functor
\newcommand{\Fil}{{\mathrm{Fil}}}           %Filtration
\newcommand{\Frob}{{\mathrm{Frob}}}         %Frobenius
\newcommand{\Gal}{{\mathrm{Gal}}}           %Galois groups
\newcommand{\Hom}{{\mathrm{Hom}}}           %Hom-functor
\newcommand{\Ima}{{\mathrm{Im}}}            %image
\newcommand{\Ker}{{\mathrm{Ker}}}           %kernal
\newcommand{\Mod}{{\mathrm{Mod}}}           %Module

\newcommand{\Rep}{{\mathrm{Rep}}}           %Representations
\newcommand{\Spf}{{\mathrm{Spf}}}           %formal affine schemes
\newcommand{\Sym}{{\mathrm{Sym}}}           %Symmetic algebra
\newcommand{\Vect}{{\mathrm{Vect}}}         %Vector bundle

\newcommand{\GL}{{\mathrm{GL}}}             % genetal linear group
\newcommand{\SL}{{\mathrm{SL}}}             % special linear group

\newcommand{\PSL}{{\mathrm{PSL}}}

\newcommand{\cris}{{\mathrm{cris}}}         % cristalline
\newcommand{\cyc}{{\mathrm{cyc}}}           % cyclotomic
\newcommand{\perf}{\mathrm{perf}}           % perfection
\newcommand{\tor}{{\mathrm{tor}}}           % torsion

\newcommand{\syn}{{\mathrm{Syn}}}
\newcommand{\Coh}{{\mathrm{Coh}}}
\newcommand{\refl}{{\mathrm{refl}}}
\usepackage{relsize}
\usepackage[bbgreekl]{mathbbol}
\DeclareSymbolFontAlphabet{\mathbb}{AMSb} %to ensure that the meaning of \mathbb does not change
\DeclareSymbolFontAlphabet{\mathbbl}{bbold}
\newcommand{\Prism}{{\mathlarger{\mathbbl{\Delta}}}} % Prism

\newcommand{\Nyg}{\mathrm{Nyg}}

\newcommand{\ad}{\mathrm{ad}}
\usepackage{titletoc}
\titlecontents{section}[1.72em]{\vspace{1pt}}%
    {\thecontentslabel.\enspace}%numbered sections
    {}%numberless section
    {\titlerule*[0.5pc]{.}\contentspage}%

\titlecontents{subsection}[1.72em]{\vspace{1pt}}%
    {\thecontentslabel.\enspace}%numbered sections
    {}%numberless section
    {\titlerule*[0.5pc]{.}\contentspage}%
\begin{document}
\title{Congruences of first syntomic cohomology groups and Bloch--Kato Selmer groups}
\author{Yu Min}
\date{}
\maketitle
\setcounter{tocdepth}{1}

\begin{abstract}
We study congruences of integral Bloch--Kato Selmer groups using congruences of prismatic $F$-gauges. Let $M$ be a reflexive $F$-gauge over the syntomification of the ring of integers of a $p$-adic field $K$ and $T$ be its \'etale realization. We prove that the first syntomic cohomology group of $M$ injects into the integral Bloch--Kato Selmer group of $T$. The cokernel is killed by a power of $p$ whose exponent depends only on the ramification degree of $K$ and the Hodge--Tate weights of $T$.

For $n$ sufficiently large, we show that a congruence between two reflexive $F$-gauges modulo $p^{2n}$ induces an isomorphism between their first syntomic cohomology groups modulo $p^n$. The bound again depends only on the ramification degree of $K$ and the Hodge--Tate weights. These results yield bounded comparison between local Bloch--Kato Selmer groups. In particular, they apply to tensorial constructions of Barsotti--Tate representations, which can go beyond the Fontaine--Laffaille range.

As an application, we present examples of congruent newforms and transfer the vanishing of Selmer groups between tensorial constructions of their associated Galois representations. We verify the rank part of the Bloch--Kato conjecture for these examples.
\end{abstract}

\tableofcontents

\section{Introduction}

Let $K$ be a finite extension of $\bQ_p$ with ring of integers $\calO_K$. Let $G_K$ denote its absolute Galois group. Recall that for any crystalline $\bQ_p$-representation $V$ of $G_K$, the (local) Bloch--Kato Selmer group is defined by
\[
H^1_f(K,V):=\Ker(H^1(K,V)\to H^1(K,V\otimes_{\bQ_p}\Bcris)).
\]

It consists of all the crystalline extensions of the trivial representation $\bQ_p$ by $V$. Let $T$ be a crystalline $\bZ_p$-representation. One can also define the integral Bloch--Kato Selmer group $H^1_f(K,T)$ as a pullback
\[
\begin{tikzcd}
    H^1_f(K,T) \arrow[r] \arrow[d]
    \arrow[dr, phantom, "\ulcorner", very near start]
    & H^1_f(K,T[1/p]) \arrow[d] \\
    H^1(K,T) \arrow[r]
    & H^1(K,T[1/p]).
\end{tikzcd}
\]
Similarly, $H^1_f(K,T)$ is the group of all crystalline extensions of the trivial representation $\bZ_p$ by $T$.

Given a global field $F$ and a $\bQ_p$-representation of the absolute Galois group $G_F$ of $F$, one can define the global Bloch--Kato Selmer group by imposing conditions on each local class, where the $p$-adic class should lie in the  local Bloch--Kato Selmer group as defined above.

The most interesting $p$-adic representations come from the $p$-adic realization of motives over $F$. Greenberg asked the question (cf. \cite[Introduction]{iovita15}): suppose we have two motives which are congruent modulo some power of $p$, what can we say about the two families of the $p$-Selmer groups?

Let us examine this question  in the local context. Let $T$ still be a crystalline $\bZ_p$-representation of $G_K$. One can consider the Galois cohomology group $H^1(K,T/p^n)$ and define the ``Selmer group of $T/p^n$ " as the subgroup $H^1_f(K,T/p^n)$ of $H^1(K,T/p^n)$ which is the image of $H^1_f(K,T)$ along the natural map $H^1(K,T)\to H^1(K,T/p^n)$. Indeed, we have $H^1_f(K,T/p^n)=H^1_f(K,T)/p^n$ (cf. \cite[Lemma 4.1.1]{iovita15}). Now suppose $T_1, T_2$ are two crystalline $\bZ_p$-representations of $G_K$ and $T_1/p^n\cong T_2/p^n$ as $G_K$-representations for some $n$. Does $T_1/p^n\cong T_2/p^n$ yield any relation between $H^1_f(K,T_1/p^n)$ and $H^1_f(K,T_2/p^n)$?  In particular, are they isomorphic? Note that the subtlety here is that the definition of $H^1_f(K,T_i/p^n)$ indeed relies on $T_i$ yet there is no direct relation between $T_1$ and $T_2$.

This question has been answered in several cases. We refer to \cite[Remark 1.14]{iovita15} for a nice summary. In particular, all the positive results are obtained under some restrictions on the ramification degree $e=e(K)$ and the Hodge--Tate weight range $[0,h]$, for example $e(h-1)\leq p-1$ as in \cite{iovita15}.

This restriction actually originates from the fact that the $G_K$-action on $T/p^n$ is too weak to capture all the information of the ``mod $p^n$ reduction" of the crystalline $\bZ_p$-representation $T$. The introduction of \cite{iovita15} has given two nice examples explaining why one cannot get positive results without certain restrictions on the ramification degrees and the Hodge--Tate weights. 

%\begin{enumerate}
%    \item  Consider $T_1=\bZ_p(1), T_2=\bZ_p(p), K=\bQ_p$ and $p\geq 3$. Then $T_1/p\cong T_2/p\cong \bF_p(1)$. But $H^1_f(\bQ_p,T_1/p)$ and $H^1_f(\bQ_p,T_2/p)$ turn out to be orthogonal with respect to the cup-product. 
 %   \item Consider two ordinary elliptic curves $E_1$ and $E_2$ over $K$, which contains all the coordinates of the $p$-torsion points $E_1[p]$ and $E_2[p]$. Hence $T_i/p=E_i[p]$ is a trivial $G_K$-module for $i=1,2$. Isomorphisms $T_1/p\cong T_2/p$ will barely give information about the relation between $H^1_f(\bQ_p,T_1/p)$ and $H^1_f(\bQ_p,T_2/p)$.
%\end{enumerate}

In this paper, we will consider another ``mod $p^n$ reduction" of crystalline $\bZ_p$-representations and try to produce an answer without imposing any restriction on the ramification degrees and the Hodge--Tate weights. Based on the local results obtained, we will then examine congruences of global Galois representations and present some concrete examples, which satisfy the rank part of the Bloch--Kato conjecture.

\subsection{Local results}

 Recall the equivalences proved in \cite{BS23} and \cite{bhatt22} .

\begin{thm}[{\cite[Theorem 6.6.13]{bhatt22}}]
There are equivalences of categories.

\begin{enumerate}
    \item The category $\Rep^{\cris}_{\bZ_p}(G_K)$ of crystalline $G_K$-representations on finite free $\bZ_p$-modules.
    \item The category $\Vect^{\varphi}((\calO_K)_{\Prism}, \calO_{\Prism})$ of prismatic $F$-crystals on $\calO_K$.
    \item The category $\Coh^{\refl}(\calO_K^\syn)$ of reflexive $F$-gauges on $\calO_K$.
\end{enumerate}
\end{thm}

Instead of reductions of crystalline representations, we will consider the mod $p^n$ reductions of the corresponding  reflexive $F$-gauges, which presumably inherit more information. To justify this replacement, we first need to show for reflexive $F$-gauges, there is a reasonable version of ``Bloch--Kato Selmer group". Note that local Bloch--Kato Selmer group classifies crystalline extensions. For $F$-gauges, the counterpart should be their first syntomic cohomology groups.

Let $M$ be a reflexive $F$-gauge of rank $d$, $\bM$ its attached prismatic $F$-crystal and $T$ its \'etale realization, whose Hodge--Tate weights are $r_1\leq r_2\leq \cdots \leq r_d$ (our convention is that the Hodge--Tate weight of $\bZ_p(1)$ is $-1$). We have the following theorem.

\begin{mainthm}\label{cohernel boundedness}
    There are injective group homomorphisms
    \[
    H^1(\calO_K^\syn, M)\xhookrightarrow{\alpha} \Ext^1(\calO_{\Prism},\bM)\xhookrightarrow{\beta} H^1_f(K,T),
    \]
such that $\Coker(\beta)$ is killed by $p^{s_1}$ for some $s_1$ only depending on the ramification degree $e(K)$ and $r_1, r_d$ and $\Coker(\alpha)$ is killed by $p^{s_2}$ for some $s_2$ only depending on $e(K)$ and the Hodge--Tate weights $r_1\leq r_2\leq \cdots \leq r_d$. Moreover there are some special cases:
\begin{enumerate}
    \item[$\bullet$] When $r_d\leq 0$, the map $\beta$ is an isomorphism.
    \item[$\bullet$] When $r_1\geq 0$, the map $\alpha$ is an isomorphism.
\end{enumerate}
\end{mainthm}
We refer to Section \ref{extension} for the explicit formulas for $s_1, s_2$ in terms of $e(K), r_1,\cdots,r_d$. The reason that we only get injections is that the equivalences of categories above are not exact. More precisely, for a short exact sequence of crystalline $\bZ_p$-representations, the corresponding sequence of Breuil--Kisin modules is left exact but not right exact in general, which accounts for $\beta$ not being necessarily surjective. For a short exact sequence of prismatic $F$-crystals, the sequence of Nygaard filtered Breuil--Kisin modules again is left exact but not right exact, which accounts for $\alpha$ not being necessarily surjective.

\begin{rmk}

\begin{enumerate}
    \item[(a)] When $r_d\leq 0$, the fact that $\beta$ is an isomorphism was proved first in \cite[Theorem A]{Coupek26}. Interestingly, the authors also proved that when $r_d\leq 0$, the map $\alpha$ is an isomorphism in \cite[Theorem A.2 and Corollary A.9]{Coupek26}.
    \item[(b)] In the unramified case, Bhatt--Lurie proved $H^1(\calO_K^\syn, M)[1/p]\cong H^1_f(K,T)[1/p]$. In the ramified case, this was proved by Pentland \cite{pentland2025syntomification}. Du and Zhao also proved the injection $H^1(\calO_K^\syn, M)\hookrightarrow H^1_f(K,T)$ and the isomorphism $H^1(\calO_K^\syn, M)[1/p]\cong H^1_f(K,T)[1/p]$ for crystalline representations with non-negative Hodge--Tate weights as a consequence of \cite[Theorems A and C]{DZ}.
\end{enumerate}
\end{rmk}

 Theorem \ref{cohernel boundedness} shows that by replacing $H^1_f(K,T)$ with  $H^1(\calO_K^\syn, M)$, we are still in a controllable situation. Our second main result shows that $H^1(\calO_K^\syn, M)/p^n$ behaves better than $H^1_f(K,T)/p^n$ when there is no restriction on the ramification degree and the Hodge--Tate weights.

 \begin{mainthm}\label{main}
 Let $M,N$ be two reflexive $F$-gauges on $\calO_K^{\syn}$ whose Hodge--Tate weights lie in $[r_1,r_d]$. Then there exists an integer $a(e,r_1)$, which only depends on the ramification degree $e$ and $r_1$, such that if there is an isomorphism $\alpha: M/p^{2n}\cong N/p^{2n}$ for some $n\geq a(e,r_1)$, one can get an isomorphism
    \[
  f_{\alpha}:  H^1(\calO_K^\syn,M)/p^{n}\xrightarrow{\sim} H^1(\calO_K^\syn,N)/p^{n},
    \]
    which is functorial in $\alpha$. 
\end{mainthm}

The integer $a(e,r_1)$ is actually a bound on the integer $n$ such that $p^nH^2(\calO_K^\syn,M)=0$. We refer to Proposition \ref{bound H^2} for an explicit formula for the bound $a(e,r_1)$.  In particular, when $r_1\geq 0$, one can take $a(e,r_1)=0$.
\begin{rmk}
    The explicit formula for $a(e,r_1)$ is obtained in \cite[Proposition 4.12]{pentland2025syntomification}. Pentland uses an explicit syntomic complex to compute the syntomic cohomology, which is inspired by \cite{antieau2024k}. It is also possible to use the explicit prismatic Herr complex  introduced in \cite{DZ} to get such a bound. In the unramified case, there is also the syntomic complex of Wach modules studied in \cite{Abhinandan}.
\end{rmk}

A quick corollary of Theorems \ref{cohernel boundedness} and \ref{main} is the following result concerning the classical finite Bloch--Kato Selmer group.
\begin{cor}[Theorem \ref{local-inclusion}]
    Let $M_1,M_2$ be two reflexive $F$-gauges over $\calO_K$ with \'etale realizations $T_1,T_2$ respectively. Then there exist integers $\omega_1,\omega_2$ which only depend on the ramification degree and Hodge--Tate weights of $T_1,T_2$ such that if there is an isomorphism $M_1/p^{2n}\cong M_2/p^{2n}$ for some $n\geq \omega_2$, with induced isomorphism $\alpha_n:H^1(K,T_1/p^n)\xrightarrow{\simeq}H^1(K,T_2/p^n)$, then there are two natural inclusions
    \[
    p^{w_1}\alpha_{n}(H^1_f(K,T_1/p^n))\subseteq H^1_f(K,T_2/p^n),
 \qquad
 p^{w_1}\alpha_{n}^{-1}(H^1_f(K,T_2/p^n))\subseteq H^1_f(K,T_1/p^n).
    \]
    
\end{cor}
%\begin{proof}
%    Let $\omega_1$ be the bound in Theorem \ref{cohernel boundedness} such that the cokernels of $H^1(\calO_K^\syn,M_1)\to H^1_f(K,T_1)$ and $H^1(\calO_K^\syn,M_2)\to H^1_f(K,T_2)$ are both killed by $p^{\omega_1}$. Then $\omega_1$ only depends on the ramification degree and Hodge--Tate weights of $T_1,T_2$. Let $\omega_2$ be the bound in Theorem \ref{main}, which again only depends on the ramification degree and Hodge--Tate weights of $T_1,T_2$. 

%    Suppose there is an integer $n\geq \omega$ such that $\alpha:M_1/p^{2n}\cong M_2/p^{2n}$. Then we have an isomorphism $f_{\alpha}:H^1(\calO_K^\syn,M_1)/p^n\cong H^1(\calO_K^\syn,M_2)/p^n$.

%    Let $\iota_i$ denote the injection $H^1(\calO_K^\syn, M_i)\to H^1_f(K,T_i)$ and $\bar\iota_{i,n}$ be its mod $p^n$ reduction. Multiplying by $p^{\omega_1}$ defines for $i=1,2$
%    \[
 %   \lambda_{i, \omega_1}: H^1_f(K,T_i)/p^n\to H^1(\calO_K^\syn,M_i)/p^n.
%    \]
%    Define $F_n:= \bar\iota_{2,n}\circ f_{\alpha}\circ \lambda_{1,\omega_1}$ and $G_n:=\bar\iota_{1,n}\circ f_{\alpha}^{-1}\circ \lambda_{2,\omega_1}$. Then $F_n\circ G_n=p^{2\omega_1}$ and $G_n\circ F_n=p^{2\omega_1}$.
    
%\end{proof}

In order to make Theorem \ref{main} more useful, we still need to address one issue: Find a way to verify two reflexive $F$-gauges have isomorphic mod $p^n$ reductions. There is no explicit description of reductions of reflexive $F$-gauges or prismatic $F$-crystals currently. However, we have the following observation:\\
 
 \textit{``The $G_K$-action" on the reductions of the associated Breuil--Kisin modules uniquely determines the reductions of the prismatic $F$-crystals. Together with the Nygaard filtrations, they can determine the reductions of the reflexive $F$-gauges}.\\

 More precisely, let $\Coh^{\refl}(\calO_K^\syn)/p^n$ be the category of mod $p^n$-reductions of reflexive $F$-gauges and $BK(\frakS/p^n)^{G_K}_{\calN}$ be the category of mod $p^n$ finite free Breuil--Kisin modules together with $G_K$-actions and Nygaardian filtrations (see Definition \ref{mod p^n gauge} and \ref{mod p^n crystalline BK}). Then we have the following theorem.
 
 \begin{mainthm}\label{main2}

The natural functor $\Coh^{\refl}(\calO_K^\syn)/p^n\to BK(\frakS/p^n)^{G_K}_{\calN}$ is fully faithful. In particular,
   for any $F_1,F_2\in \Coh^{\refl}(\calO_K^\syn)$, there is an isomorphism $F_1/p^n\simeq F_2/p^n$ if and only if there is an isomorphism of the mod $p^n$-reductions of the associated Breuil--Kisin modules with $G_K$-actions and Nygaard filtrations.
\end{mainthm}

Theorem \ref{main2} actually implies that congruences of reflexive $F$-gauges impose strict control on their Hodge--Tate weights: If two reflexive $F$-gauges $M_1, M_2$ are congruent mod $p^n$ for some $n$, their Hodge--Tate weights must be the same (cf. Proposition \ref{HT weights}).

Still, the criterion in Theorem \ref{main2} is hard to verify in practice. It is expected that congruences of Galois representations are equivalent to  congruences of reflexive $F$-gauges in the usual small ramification-weight case $eh<p-1$. For example, for abelian schemes over $\calO_K$ with $e(K)<p-1$, it suffices to check the congruences of their Tate modules. There is no need to use prismatic $F$-gauges in these cases at all: The finite Bloch--Kato Selmer groups can be defined using the fppf cohomology of their associated finite flat group schemes.

However, we can always take some tensorial constructions of Barsotti--Tate representations (for example symmetric power) which will go beyond the Fontaine--Laffaille range easily. Hence the finite Bloch--Kato Selmer groups of these tensorial constructions are not determined by the torsion Galois representations. This is the place where prismatic $F$-gauges can be useful. Note that the prismatic $F$-gauges attached to Barsotti--Tate representations are indeed vector bundles and any tensorial constructions of vector bundles remain vector bundles. Congruences of Barsotti--Tate representations are equivalent to congruences of their attached prismatic $F$-gauges, which can further induce congruences of their tensorial constructions. Then we can apply the above local results to these tensorial constructions of $F$-gauges. This observation leads to the concrete global applications in the next section.

\subsection{Global applications}
The local congruence result can induce the following global result once other necessary primes ($\neq p$) are also taken into account.
\begin{mainthm}[Theorem \ref{global-bound}]\label{main-global}
Let $T_1,T_2$ be finite free $\bZ_p$-representations of
$G_{F,\Sigma}$, crystalline at every $v\mid p$ and suppose that
there is a global congruence
\[
 \alpha_{2n}:T_1/p^{2n}\xrightarrow{\sim}T_2/p^{2n}.
\]
For each $v\mid p$, assume that $T_i|_{G_{F_v}}$ is the
\'{e}tale realization of a reflexive $F$-gauge
$N_{i,v}$ and that there is a compatible congruence
\[
 \gamma_{v,2n}:N_{1,v}/p^{2n}
 \xrightarrow{\sim}N_{2,v}/p^{2n}
\]
whose \'{e}tale realization is $\alpha_{2n}|_{G_{F_v}}$.  Assume that $n\geq a(e(F_v),r_{1,v})$ as in Theorem \ref{main} where $r_{1,v}$ is the smallest Hodge--Tate weight and let $c_v$ be an exponent for the composite injection in Theorem \ref{cohernel boundedness} for the two local reflexive $F$-gauges (i.e. $s_1+s_2$ in Theorem \ref{cohernel boundedness} ).

For $v\nmid p$, let $b_v(T_i)$ be the smallest exponent such that  $p^{b_v(T_i)}H^1(I_v,T_i)_{\tor}=0$.  Let $\alpha_n$ be the reduction
of $\alpha_{2n}$ modulo $p^n$ and put
\[
 D=\max\{\max_{v\mid p}c_v,\ \max_{\substack{v\in\Sigma_{\rm fin}\\v\nmid p}}
 (b_v(T_1)+b_v(T_2))
 \}.
\]

Then we have natural inclusions
\[
 p^{D}\alpha_{n,*}H_f^1(F,T_1/p^n)
 \subseteq H_f^1(F,T_2/p^n),
 \qquad
 p^{D}\alpha_{n,*}^{-1}H_f^1(F,T_2/p^n)
 \subseteq H_f^1(F,T_1/p^n).
\]
\end{mainthm}

As mentioned above, although the local condition is hard to verify in general, we can still apply this global result to some interesting cases, notably tensorial constructions of Barsotti--Tate representations. A global example is as follows.

Consider the newform $f=43.2.\mathrm c.\mathrm b$ with nebentype $\epsilon$. For each positive integer $r$, the deformation theory in \cite{CP18} produces another newform $g_r$ such that $\rho_f\equiv \rho_{g_r}\mod 31^r$ (Proposition \ref{CP-targets}) and one can apply Theorem \ref{main-global} to 
\[
S_f^{(30)}=\End(\Sym^{30}(\rho_f))\otimes\epsilon^{-1},\qquad
 S_{g,N}^{(30)}=
 \Sym^{30}(\rho_{g_{2N}})\otimes(\Sym^{30}(\rho_f))^\vee
 \otimes\epsilon^{-1}.
 \]

 The symmetric power $30$ is designed to go beyond the Fontaine--Laffaille range. \cite{AHTW26} guarantees the source vanishing $H_f^1(F_m,S_f^{(30)}[1/p])=0$ where $F_m=\bQ(\mu_{31^m})$ for each $m\geq 1$. Then  combining Theorem \ref{main-global} with some analytic study of $L$-functions, one can get the following.

\begin{cor}[Corollary \ref{BK}]
For every character $\eta:\Delta_m:=\Gal(F_m/\bQ)\to E^\times$, put \[W_{N,\eta}^{(30)}=
 (S_{g,N}^{(30)}[1/31]\otimes_{\bQ_{31}}E)\otimes\eta^{-1}\]
 where $E$ is a large enough coefficient field. For any $N$ greater than a constant only depending on $f$ and $m$, we have
\begin{equation*}\label{rank-part-three-vanishings}
 H_f^1(\bQ,W^{(30)}_{N,\eta})=0,
 \qquad H^0(\bQ,W^{(30)}_{N,\eta})=0,
 \qquad
 {\rm ord}_{s=0}
 L(W^{(30)*}_{N,\eta}(1),s)=0.
\end{equation*}
Hence the rank part of the Bloch--Kato conjecture holds
\[
 {\rm ord}_{s=0}L(W^{(30)*}_{N,\eta}(1),s)=\dim_E H_f^1(\bQ,W^{(30)}_{N,\eta})-\dim_EH^0(\bQ,W^{(30)}_{N,\eta}).
\]
\end{cor}

This is just one typical example one can use to demonstrate the significance of Theorem \ref{main-global}. However in this example, the impact of ramified primes $l\neq p=31$ has been reduced to the minimal level by design. We expect that by seriously taking into account the impact of ramified primes, one can get many more interesting examples.

\section*{Acknowledgement}
We would like to thank Yupeng Wang for pointing out an error in an early draft and Heng Du for his interest. We are grateful to Yichao Tian for asking the question concerning the boundedness in Theorem \ref{cohernel boundedness} and \ref{main}. Special thanks go to Dylan Pentland for pointing us to his work \cite{pentland2025syntomification} which makes Theorem \ref{main} unconditional. We also want to thank Pavel Čoupek, Evangelia Gazaki and Adriano Marmora for helpful correspondence. The author is supported by the Hong Kong RGC ECS Grant 26310126.

\section*{AI disclosure}
The global example in Section \ref{global ex} was constructed through interactions with ChatGPT $5.6$ Sol. Note that Theorem \ref{main-global} is particularly useful to prove the vanishing of Bloch--Kato Selmer groups for sufficiently deep level. Based on the recent work \cite{AHTW26}, we asked ChatGPT to find examples of congruences where the source can be covered by \cite{AHTW26} while the target cannot be covered by \cite{AHTW26}, using tensorial constructions of Barsotti--Tate representations. Also, we require that the target is outside the usual small ramification-weight case and is non-ordinary. ChatGPT 5.6 Sol then found the key reference \cite{CP18}, which can produce such targets via the deformation theory. The main task is to verify that the chosen source satisfies the assumptions in \cite{CP18}, \cite{AHTW26}. The original arguments in Section \ref{global ex} were due to ChatGPT $5.6$ Sol and were verified by us. We rewrote the arguments, adding details or simplifying the exposition as appropriate. In addition, the results in Subsection \ref{Galois coh} were indeed extracted from our discussion with ChatGPT about the global example and we present them in a general way. We are responsible for the correctness of all these arguments.

\section{Mod $p^n$ reductions}
We will prove Theorem \ref{main} and \ref{main2} in this section. Let $M$ be a reflexive $F$-gauge over $\calO_K^\syn$ with \'etale realization $T$, whose Hodge--Tate weights are $r_1\leq r_2\leq \cdots \leq r_d$. When $K$ is unramified, Bhatt--Lurie proved in \cite{bhatt22} that $H^1(\calO_K^\syn,M)[1/p]=H^1_f(K,T)[1/p]$ and $H^2(\calO_K^\syn,M)[1/p]=0$. This has been generalized by Pentland to the ramified case.

\begin{thm}[{\cite[Proposition 4.8, Lemma 4.11, Proposition 4.12]{pentland2025syntomification}}]\label{syntomic cohomology}
Let $K$ be a finite extension of $\bQ_p$. Let $M$ be a reflexive $F$-gauge over $\calO_K^\syn$ with \'etale realization $T$. Then the syntomic cohomology $R\Gamma(\calO_K^\syn,M)$ is calculated by the total fiber of the square
\begin{equation*}
    \xymatrix{
F^\bullet\Fil^0_{\Nyg}\varphi^*\frakM\ar[rr]^{\Fil^0_{\Nyg}\nabla_\frakM}\ar[d]^{can-\varphi_{\frakM}} & &F^{\bullet-1}\Fil^{-1}_{\Nyg}\varphi^*\frakM\ar[d]^{can-\varphi_{\frakM}^{\nabla}}\\
    F^\bullet\varphi^*\frakM\ar[rr]^{\nabla_{\frakM}}& &F^{\bullet-1}\varphi^*\frakM
    }
\end{equation*}
where $\frakM$ is the attached Breuil--Kisin module  and $\Fil^\bullet_{\Nyg}\varphi^*\frakM:=\{x\in \varphi^*\frakM\mid \varphi(x)\in E^\bullet\frakM\}$ (we refer to \cite{pentland2025syntomification} for all the other notations). In particular, we have
\begin{enumerate}
    \item $H^1(\calO_K^\syn,M)[1/p]=H^1_f(K,T)[1/p]$.
    \item $H^2(\calO_K^\syn,M)[1/p]=0$.
    \item $H^{\geq 3}(\calO_K^\syn,M)=0$.
\end{enumerate}
\end{thm}

We  begin with a refined statement about $H^2(\calO_K^\syn,M)$ due to Pentland.

\begin{prop}[{\cite[Proposition 4.12]{pentland2025syntomification}}]\label{bound H^2}
    There exists an integer $a(e,r_1)$ which only depends on the ramification degree $e$ and $r_1$ such that $p^{a(e,r_1)}H^2(\calO_K^\syn,M)=0$.
\end{prop}
\begin{proof}
    This is  proved in the proof of \cite[Proposition 4.12]{pentland2025syntomification}. The task is just to make the integer $a(e,r_1)$ explicit.  We refer to \textit{loc. cit} for all the notations. When $r_1\geq 0$, \cite[Proposition 4.12]{pentland2025syntomification} already proved $H^2(\calO_K^\syn,M)=0$. So we assume $r_1<0$.

    Note that $H^2(\calO_K^\syn,M)$ is the cokernel of the following map
    \begin{equation}\label{coker}
    F^{\bullet-1}\Fil^{-1}_{\Nyg}\varphi^*\frakM\oplus F^\bullet\varphi^*\frakM\xrightarrow{\begin{pmatrix} can-\varphi_\frakM^{\nabla}\\  -\nabla_\frakM\end{pmatrix}^T} F^{\bullet-1} \varphi^* \frakM.
    \end{equation}

    The filtrations $F^\bullet$ are all $u$-adic filtrations, i.e. $F^\bullet(-)=u^{\bullet}(-)$ and $F^{i}=F^0$ if $i\leq 0$. Let $\alpha$ denote the above map. Note that we have the Frobenius semilinear map
    \[
    \varphi_{\frakM}:\frakM\to \frac{1}{E^{-r_1}}\frakM.
    \]

     Pentland proved that there exists an integer $i_0$ such that for any $i\geq i_0$, the map $can-\varphi_\frakM^{\nabla}$ is a well-defined endomorphism of $F^i\varphi^*\frakM/p$. In fact, 
    we have
    \begin{equation}\label{phi-nabla}
\varphi^\nabla_\frakM:F^{\bullet}\Fil^{-1}_{\Nyg}\varphi^*\frakM=u^{\bullet}\Fil^{-1}_{\Nyg}\varphi^*\frakM\to F^{p\cdot \bullet -1}\varphi^*\frakM=u^{p\cdot \bullet -1}\varphi^*\frakM
    \end{equation}
as defined in \cite[Definition 4.54]{antieau2024k}. We first claim there is a non-negative integer $x$ such that $u^x\varphi^*\frakM/p\subseteq \Fil^{-1}_{\Nyg}\varphi^*\frakM/p$. Using $E\equiv u^e\mod p$, it suffices to have $x+er_1\geq -e$. So we can take $x=e(-r_1-1)$. 

By the map (\ref{phi-nabla}), to find $i_0$ is equivalent to finding an integer $y$ such that
\[
x+y\leq py-1.
\]
We can take $y=\lceil \frac{x+1}{p-1}\rceil$. Then $i_0=x+y=\lceil\frac{pe(-r_1-1)+1}{p-1}\rceil$. Let $F^\bullet \Coker(\alpha)$ be the filtration induced by $F^\bullet\varphi^*\frakM$. Note that $F^\bullet \Coker(\alpha)$ is not $u$-adic filtration as the maps involved in \ref{coker} are not $\frakS$-linear. Then  Pentland shows $F^i\Coker(\alpha)=0$ for all $i\geq i_0$.

It remains to investigate $F^i\Coker(\alpha)$ for $0\leq i<i_0$. We look at the other map $\nabla_{\frakM}:F^\bullet\varphi^*\frakM\to F^{\bullet-1}\varphi^*\frakM$. The paragraph before \cite[Proposition 4.12]{pentland2025syntomification} shows that $\gr^{i+1}_{F}\nabla_{\frakM}=(i+1)\cdot u_{i+1}$ where $u_{i+1}$ is a $p$-adic unit. So $\gr^{i}_F(\Coker(\alpha))$ is killed by $p^{v_p(i+1)}$. 

Finally, we can take \[
a(e,r_1)=\begin{cases}
0 & \text{if }\ r_1\geq 0 \\
    \sum_{i=0}^{\lceil\frac{pe(-r_1-1)+1}{p-1}\rceil-1}v_p(i+1) & \text{if}\ r_1<0.
\end{cases}
\]

\end{proof}

Now we are ready to prove Theorem \ref{main}.

\begin{proof}[Proof of Theorem \ref{main}]
Note that $M$ is $p$-torsion free. There is a  short exact sequence for each $m$
\[
0\to M\xrightarrow{\times p^m}M\to M/p^m\to 0. 
\]

Then there is a long exact sequence
{\small{\[
 0\to H^1(\calO_K^\syn, M)/p^m\to H^1(\calO_K^\syn, M/p^m)\xrightarrow{\beta} H^2(\calO_K^\syn, M)\xrightarrow{\times p^m}H^2(\calO_K^\syn, M)\xrightarrow{\gamma} H^2(\calO_K^\syn, M/p^m) \to H^3(\calO_K^\syn, M)=0.
\]}}

The vanishing $H^3(\calO_K^\syn, M)=0$ follows from Theorem \ref{syntomic cohomology}. Let $a(e,r_1)$ be the integer in Proposition \ref{bound H^2}. Then we have $p^{a(e,r_1)}H^2(\calO_K^\syn,M)=p^{a(e,r_1)}H^2(\calO_K^\syn,N)=0$. Hence we get a short exact sequence for any $n\geq a(e,r_1)$
\begin{equation}\label{key}
0\to H^1(\calO_K^\syn, M)/p^{n}\to H^1(\calO_K^\syn, M/p^{n})\xrightarrow{\gamma\circ \beta} H^2(\calO_K^\syn,M/p^{n})\to 0.
\end{equation}
More naturally, the composite $\gamma\circ\beta$ is nothing but the Bockstein connecting homomorphism attached to the short exact sequence
\[0\to M/p^n\xrightarrow{\times p^n}M/p^{2n}\to M/p^n\to 0.\]

To see this, consider the commuting diagram
\begin{equation*}
    \xymatrix{
    0\ar[r]& M\ar[r]^{\times p^n}\ar[d]& M\ar[r]\ar[d]&M/p^n\ar[d]^{id}\ar[r]& 0\\
    0\ar[r]& M/p^n\ar[r]^{\times p^n}& M/p^{2n}\ar[r]& M/p^n\ar[r]& 0. 
    }
\end{equation*}
This induces a commutative diagram
\begin{equation*}
    \xymatrix{
    H^1(\calO_K^\syn, M/p^n)\ar[r]^{\beta}\ar[d]^{id}& H^2(\calO_K^\syn,M)\ar[d]^{\gamma}\\
    H^1(\calO_K^\syn,M/p^n)\ar[r]^{\beta_B}& H^2(\calO_K^\syn,M/p^n)
    }
\end{equation*}
where $\beta_B$ is just the Bockstein connecting homomorphism.

For any isomorphism $\alpha: M/p^{2n}\cong N/p^{2n}$ for some $n\geq a(e,r_1)$, we then get an isomorphism induced by the short exact sequence \ref{key}
\[
f_{\alpha}:H^1(\calO_K^\syn,M)/p^n\cong H^1(\calO_K^\syn,N)/p^n.
\]

The functoriality in $\alpha$ results from the functoriality of Bockstein homomorphism.
\end{proof}

 Next we investigate how to verify $M/p^{2n}\cong N/p^{2n}$ and explain the  observation in the introduction in detail. We first deal with the case of prismatic $F$-crystals. Let $(\frakS=W(k)[[u]],E)$ be the Breuil--Kisin prism in $(\calO_K)_{\Prism}$, where $k$ is the residue field of $\calO_K$ and $E$ is the Eisenstein polynomial of a fixed uniformizer $\pi$ of $\calO_K$. Recall that there is an equivalence between crystalline $\bZ_p$-representations and crystalline Breuil--Kisin modules (cf. \cite[Definition 1.1.8 and Proposition 7.1.10]{gao2023breuil}).
 
 \begin{dfn}\label{crystalline BK}
     A finite free Breuil--Kisin module $\frakM$ of height $[0,r]$ equipped with a $G_K$-action on $\frakM_{\inf}:=\frakM\otimes_{\frakS}\Ainf$ is a crystalline Breuil--Kisin module if it satisfies the following conditions:
     \begin{enumerate}
         \item $\frakM\subseteq \frakM_{\inf}^{G_{\infty}}$ where $G_{\infty}$ is the Galois group $\Gal(\bar K/K_{\infty})$ where $K_{\infty}=\cup_n K(\pi_n)$ and $\{\pi_n\}_n$ is a compatible family of $p^n$-th root of $\pi$.
         \item $\frakM/u\subseteq (\frakM_{\inf}/W(\frakm^{\flat})\frakM_{\infty})^{G_K}$ where $\frakm^\flat$ is the maximal ideal of $\calO_C^{\flat}$.
         \item For any $g\in G_K$, we have $(g-1)(\frakM)\subseteq \varphi^{-1}(\mu)[\pi^\flat]\frakM_{\inf}$, where $\mu=[\epsilon]-1$ for a compatible family of roots of unity $\epsilon=(1,\zeta_p,\zeta_{p^2},\cdots)\in \calO_{C}^\flat$ and $[\pi^\flat]$ is the image of $u$ under the map $\frakS\to \Ainf$.
     \end{enumerate}
 \end{dfn}

  Note that $\frakM$ is not stable under the $G_K$-action. As we can see, the category of crystalline $\bZ_p$-representations is already a full subcategory of the category of Breuil--Kisin modules with $G_K$-actions. In the mod $p^n$-case, let $BK(\frakS/p^n)^{G_K}$ denote the category of finite free Breuil--Kisin modules $(\frakM,\varphi)$ over $\frakS/p^n$ with $G_K$-actions on $\frakM_{\inf}:= \frakM\otimes_{\frakS}\Ainf$. Then we have the following theorem.

\begin{thm}\label{key1}
    The category $\Vect^{\varphi}((\calO_K)_{\Prism},\calO_{\Prism}/p^n)$ of prismatic $F$-crystals in vector bundles over $\calO_{\Prism}/p^n$ is a full subcategory of $BK(\frakS/p^n)^{G_K}$.
\end{thm}

\begin{proof}
    Let $(\frakS^{\bullet},(E))$ be the \v Cech nerve of $(\frakS,(E))$ in $(\calO_K)_{\Prism}$. Then we have
    \begin{equation}
    \xymatrix@=0.5cm{
    \Vect^{\varphi}((\calO_K)_{\Prism},\calO_{\Prism}/p^n)\ar@<.0ex>[r]^-{\simeq} & \varprojlim (\Vect^{\varphi}(\frakS/p^n)\ar@<-.3ex>[r]\ar@<.3ex>[r] & \Vect^{\varphi}(\frakS^1/p^n)\ar@<-.6ex>[r]\ar@<-.0ex>[r]\ar@<.6ex>[r]&\Vect^{\varphi}(\frakS^2/p^n))}
  \end{equation}

  Write $(\Ainf,(\xi))$ for the Fontaine prism. Now consider the prism $(C(G_K,\Ainf), (\xi))\in (\calO_K)_{\Prism}$. There is a natural map\footnote{The $(\Ainf,(\xi))\times (\Ainf,(\xi))$ indeed means the prism corresponding to the self-product of the object $(\Spf(\calO_K)\leftarrow\Spf(\calO_C)\to\Spf(\Ainf)$.} $(\Ainf,(\xi))\times (\Ainf,(\xi))\to (C(G_K,\Ainf), (\xi))$ where one map $(\Ainf,(\xi))\to (C(G_K,\Ainf), (\xi))$ is induced by the constant map and the other map $(\Ainf,(\xi))\to (C(G_K,\Ainf), (\xi))$ is induced by the $G_K$-action on $\Ainf$. Hence we also have a natural map $(\frakS^1,E)\to (C(G_K,\Ainf), (\xi))$. Similarly, we have a natural map $(\frakS^2,E)\to (C(G_K\times G_K,\Ainf), (\xi))$

  Then the category $BK(\frakS/p^n)^{G_K}$ of Breuil--Kisin modules over $\frakS/p^n$ with $G_K$-actions is just

    \begin{equation}
    \xymatrix@=0.5cm{
      BK(\frakS/p^n)^{G_K}\ar@<.0ex>[r]^-{\simeq} & \varprojlim (\Vect^{\varphi}(\frakS/p^n)\ar@<-.3ex>[r]\ar@<.3ex>[r] & \Vect^{\varphi}(C(G_K,\Ainf/p^n))\ar@<-.6ex>[r]\ar@<-.0ex>[r]\ar@<.6ex>[r]&\Vect^{\varphi}(C(G_K\times G_K,\Ainf/p^n))}.
  \end{equation}

 Note that we have natural morphisms $\pi_i: \Vect^{\varphi}(\frakS^i/p^n)\to \Vect^{\varphi}(C(G_K^i,\Ainf/p^n))$ for $i=0,1,2$, which induces a morphism
 \[
 \pi: \Vect^{\varphi}((\calO_K)_{\Prism},\calO_{\Prism}/p^n)\to BK(\frakS/p^n)^{G_K}.
 \]

 We aim to show $\pi$ is fully faithful. Choose any $\bM,\bN\in \Vect^{\varphi}((\calO_K)_{\Prism},\calO_{\Prism}/p^n)$. We need to show $\Hom(\bM,\bN)\cong \Hom(\pi(\bM),\pi(\bN))$. Using the cosimplicial interpretations of both categories, this directly follows from that the natural map 
 \[\Hom(\bM(\frakS^1),\bN(\frakS^1))\to \Hom(\bM(C(G_K,\Ainf)),\bN(C(G_K,\Ainf)))\] is injective, where all the Hom's are morphisms compatible with Frobenius. To see this injectivity,  note that it embeds into the injective composite 
  \[\Hom(\bM(\frakS^1)[\frac{1}{E}],\bN(\frakS^1)[\frac{1}{E}])\to \Hom(\bM(A_\infty^1)\frac{1}{E},\bN(A_\infty^1)[\frac{1}{E}])\to \Hom(\bM(C(G_K,\Ainf))[\frac{1}{\xi}],\bN(C(G_K,\Ainf))[\frac{1}{\xi}])\]
  where all the Hom's are taken in the category of \'etale $\varphi$-modules. Here $(A_\infty,E)$ is the perfection of the Breuil--Kisin prism $(\frakS,E)$ and $(A_\infty^1,E)$ is the product $(A_\infty,E)\times (A_\infty,E)$ in the perfect prismatic site $(\calO_K)_{\Prism}^{\perf}$, hence is the perfection of $(\frakS^1,E)$. So the first map in the above composition is indeed a bijection by a mod $p^n$-version of \cite[Theorem 4.6]{wu2021galois}. The second map in the above composition is injective due to the injectivity of $A_\infty^1/p^n[1/E]\to C(G_K,\Ainf/p^n)[1/\xi]$. In fact, let $(A_{\inf}^1,\xi)$ be the self-product of $(\Ainf,\xi)$ in $(\calO_K)_{\Prism}^{\perf}$. Then $A_\infty^1\to A^1_{\inf}$ is $(p,E)$-completely faithfully flat, which particularly implies $A_\infty^1/p^n\to A^1_{\inf}/p^n$ is injective. Note that $C(G_K,\Ainf/p^n)[1/\xi]=A_{\inf}^1/p^n[1/\xi]$ by \cite[Lemma 5.3]{wu2021galois}. So we are done.
\end{proof}
  
\begin{rmk}
    \begin{enumerate}
        \item  Kisin proved that there is a fully faithful functor from crystalline $\bZ_p$-representations to finite free Breuil--Kisin modules, but it is necessary to include the $G_K$-actions on the Breuil--Kisin modules in the mod $p^n$ case.
        \item We can also work with Wach modules in the unramified case. Then the Galois action is stable on the Wach modules and there is no need to base change to $\Ainf$.
    \end{enumerate}
\end{rmk}

Now we study the mod $p^n$-reductions of  reflexive $F$-gauges on  $\calO_K^\syn$.

    Let $R^\bullet$ be the $p$-completed \v Cech nerve of the map $\calO_K\to \calO_C$.  Consider $F\in \Coh^{\refl}(\calO_K^\syn)$, whose pullback to $\Coh((R^i)^\syn)$ corresponds to the triple $(M(F^i),\Fil^\bullet M(F^i),\tilde \varphi_{M(F^i)})$ as in \cite[Example 6.1.7]{bhatt22}. Note that each $M(F^i)$ is a vector bundle over $\Prism_{R^i}$ and $\Fil^\bullet M(F^i)$ is a saturated Nygaardian filtration of the $F$-crystal $M(F^i)$ over $\Prism_{R^i}$. Then $(M(F^i)/p^n,\Fil^\bullet M(F^i)/p^n,\tilde \varphi_{M(F^i)/p^n})$ corresponds to the pullback of $F/p^n\in \Coh(\calO_K^\syn)$ to $\Coh((R^i)^\syn)$. 
    
    \begin{lem}\label{mod p^n injective}
         The map $\Fil^\bullet M(F^i)/p^n\to M(F^i)/p^n$ is injective.
    \end{lem}
    
    \begin{proof}
    Suppose there exists $x\in M(F^i)$ such that $p^nx\in \Fil^\bullet M(F^i)$. Then under the map
    \[
    M(F^i)\to \varphi^*M(F^i)\subseteq \varphi^*M(F^i)[\frac{1}{I}]\cong M(F^i)[\frac{1}{I}],
    \]
    we get $p^n\varphi(x)\in I^\bullet M(F^i)$, where $(I)$ is the divisor of $\Prism_{R^i}$. Suppose $\varphi(x)\in I^aM(F^i)$ for some $a<\bullet$. Then this induces a $p^n$-torsion in $I^aM(F^i)/I^\bullet M(F^i)$. However, $I^aM(F^i)/I^\bullet M(F^i)$ is $p$-torsion free as each $\Prism_{R^i}/I$ is $p$-torsion free by \cite[Proposition 7.10]{BS22}. Thus $\Fil^\bullet M(F^i)/p^n\to M(F^i)/p^n$ is injective.
        
    \end{proof}

    However $\Fil^\bullet M(F^i)/p^n\to M(F^i)/p^n$ is not saturated in general. So the Nygaardian filtration on $M(F^i)/p^n$ is not determined by the prismatic $F$-crystal $M(F^i)/p^n$. 

    Let $\Mod_{fp}^\varphi(\Prism_{R^i})$ be the category of prismatic $F$-crystals $(\bM,\varphi)$ in finitely presented modules over $\Prism_{R^i}$ . Consider the category $\Mod^\varphi_{fp}(\Prism_{R^i})^+ $ of prismatic $F$-crystals $(\bM,\varphi)\in \Mod^{\varphi}_{fp}(\Prism_{R^i})$ together with a Nygaardian filtration (cf. \cite[Definition 6.6.6]{bhatt22}).

Note that $\Fil^\bullet M(F^i)\widehat\otimes^\bL_{\Prism_{R^i}}\Prism_{R^j}\cong \Fil^\bullet M(F^j)$ for any map $R^i\to R^j$ in the cosimplicial ring $R^\bullet$ by \cite[Lemma 6.6.10]{bhatt22}, where the completion means derived $(p,I)$-completion. 

\begin{lem}\label{complete}
    We have $\Fil^\bullet M(F^i)/p^n\widehat\otimes^\bL_{\Prism_{R^i}}\Prism_{R^j}\cong \Fil^\bullet M(F^j)/p^n$, where the completion is derived $I$-adic completion and $\Fil^\bullet M(F^j)/p^n$ is even classically $I$-adic complete.
\end{lem}

\begin{proof}
Consider the short exact sequence
    \[
    0\to p^n\Fil^\bullet M(F^j)\to \Fil^\bullet M(F^j)\to \Fil^\bullet M(F^j)/p^n\to 0.
    \]
    As both $p^n\Fil^\bullet M(F^j)$ and $\Fil^\bullet M(F^j)$ are derived $I$-adic complete, so is $\Fil^\bullet M(F^j)/p^n$. This directly shows $\Fil^\bullet M(F^i)/p^n\widehat\otimes^\bL_{\Prism_{R^i}}\Prism_{R^j}\cong \Fil^\bullet M(F^j)/p^n$.

    Note that $M(F^j)$ is finite projective over $\Prism_{R^j}$. So $M(F^j)/p^n$ is finite projective over $\Prism_{R^j}/p^n$. As $(\Ainf,\xi)\to (\Prism_{R^j},I)$ is $(p,\xi)$-completely faithfully flat, we see $\Prism_{R^j}/p^n$ is $I$-torsion free and $M(F^j)/p^n$ is then also $I$-torsion free. 
    
    Since $\Fil^\bullet M(F^j)/p^n\to M(F^j)/p^n$ is injective by Lemma \ref{mod p^n injective}, we see that $\Fil^\bullet M(F^j)/p^n$ is $I$-torsion free, which then implies $\Fil^\bullet M(F^j)/p^n$ is classically $I$-adic complete.
\end{proof}

Consider $F_1,F_2\in \Coh^{\refl}(\calO_K^\syn)$. Let $\bF_1, \bF_2$ be their attached prismatic $F$-crystals. Then we have the following result.
\begin{prop}\label{key-crystal}
There is an isomorphism
\[
\Hom(F_1/p^n,F_2/p^n)\cong \Hom_{\Fil}(\bF_1/p^n,\bF_2/p^n)
\]
where $\Hom_{\Fil}(\bF_1/p^n,\bF_2/p^n)$ is the group of all 
 morphisms between $\bF_1/p^n, \bF_2/p^n$ which are compatible with the Nygaardian filtrations $\Fil^\bullet M(F_1^i)/p^n$, $\Fil^\bullet M(F_2^i)/p^n$.
\end{prop}
\begin{proof}
 This directly follows from the description of coherent sheaves on $(R^i)^\syn$ as triples \[(M(F^i), \Fil^\bullet M(F^i), \tilde\varphi_{M(F^i)}:\Fil^\bullet M(F^i)\to I^\bullet M(F^i))\] as in \cite[Example 6.1.7 and 6.6.7]{bhatt22} and that $\Fil^\bullet M(F^j)/p^n\to M(F^j)/p^n$ are all injective.

\end{proof}

Let $\bF$ be the prismatic $F$-crystal attached to $F\in \Coh^{\refl}(\calO_K^\syn)$. Then by \cite[lemma 6.6.10]{bhatt22}, the saturated Nygaardian filtrations $\Fil^\bullet M(F^\bullet)$ define prismatic crystals $\Fil^\bullet \bF$, which admit maps to $\bF$. Note that the prismatic $F$-crystals $M(F^\bullet)$ over $\Prism_{R^\bullet}$ are just $\bF(\Prism_{R^\bullet},I)$. A priori, $\Fil^\bullet \bF$ are just crystals in $(p,I)$-complete complexes. In particular, evaluating on the Breuil--Kisin prism $(\frakS,E)$, we get derived $(p,E)$-complete $\frakS$-complexes $\Fil^\bullet \bF(\frakS,E)$. However, $\Fil^\bullet \bF(\frakS,E)$ need not be discrete. Instead of studying $\Fil^\bullet \bF(\frakS,E)$, we turn to study the filtration on $\varphi^*\bF(\frakS,E)$, which is the classical Nygaard filtration on Breuil--Kisin modules.

 Now recall the following classical definition of Nygaard filtration on Breuil--Kisin modules.
\begin{dfn}\label{Classical Nygaard fil}
    Let $(\frakM,\varphi)$ be a finite free Breuil--Kisin module. For any $i\in \bZ$, the $i$-th Nygaard filtration on $\varphi^*\frakM$ is defined as follows
    \[
    \calN^i\varphi^*\frakM:=\{x\in \varphi^*\frakM \mid  \Phi(x)\in E^i\frakM, \Phi:(\varphi^*\frakM)[1/E]\xrightarrow{\cong}\frakM[1/E]\}
    \]
\end{dfn}

 Let us continue to look at the mod $p^n$ case. For $F\in \Coh^{\refl}(\calO_K^\syn)$, Lemma \ref{complete} yields objects \[(M(F^\bullet)/p^n,\varphi_{M(F^\bullet)/p^n},\Fil^\bullet M(F^\bullet)/p^n)\] in the cosimplicial category $\Mod^\varphi_{fp}(\Prism_{R^\bullet})^+ $, which corresponds to prismatic crystals $\Fil^\bullet\bF/p^n$ in $I$-complete complexes over $\calO_{\Prism}/p^n$. Again, we have natural maps $\Fil^\bullet\bF/p^n\to \bF/p^n$. Note that this induced Nygaardian filtration $\Fil^\bullet\bF/p^n$ might not be saturated in general, which means it is not determined by $\bF/p^n$.

We can also define the classical Nygaard filtrations $\calN^\bullet\varphi^*\bF$ on $\varphi^*\bF$ as in Definition \ref{Classical Nygaard fil} and take their mod $p^n$ reductions $\calN^\bullet\varphi^*\bF/p^n$. As $\varphi: \Ainf\to \Ainf$ is an isomorphism, we see that the natural map $\Fil^\bullet\bF(\Ainf,\xi)\to \calN^\bullet\varphi^*\bF(\Ainf,\xi)$ is an isomorphism of groups.

\begin{prop}\label{key2}
   Let $F_1,F_2\in \Coh^{\refl}(\calO_K^\syn)$. Then there is an isomorphism
   \[
   \Hom(F_1/p^n,F_2/p^n)\cong \Hom_{\calN}(\bF_1/p^n,\bF_2/p^n)
   \]
where $\Hom_{\calN}(\bF_1/p^n,\bF_2/p^n)$ is the group of all morphisms between $\bF_1/p^n, \bF_2/p^n$ which are compatible with the mod $p^n$ reductions of the classical Nygaard filtrations $\calN^\bullet\varphi^*\bF_1(\frakS,E)/p^n, \calN^\bullet\varphi^*\bF_2(\frakS,E)/p^n$. 
\end{prop}

\begin{proof}
It is easy to see that the induced maps $\calN^\bullet\varphi^*\bF_i(\frakS,E)/p^n\to \varphi^*\bF_i(\frakS,E)/p^n$ are injective and $\calN^\bullet\varphi^*\bF_i(\frakS,E)/p^n\otimes_{\frakS}\Ainf\cong \calN^\bullet \varphi^*\bF_i(\Ainf,\xi)$ for $i=1,2$. Together with the isomorphisms $\Fil^\bullet\bF_i(\Ainf,\xi)/p^n\to \calN^\bullet\varphi^*\bF_i(\Ainf,\xi)/p^n$, we see that any map $\calN^\bullet\varphi^*\bF_1(\frakS,E)/p^n\to \calN^\bullet\varphi^*\bF_2(\frakS,E)/p^n$ will induce a map $\Fil^\bullet\bF_1/p^n\to \Fil^\bullet\bF_2/p^n$. Then the statement follows from Proposition \ref{key-crystal}.
\end{proof}

\begin{dfn}\label{mod p^n gauge}
    Let $\Coh^{\refl}(\calO_K^\syn)/p^n$ denote the full subcategory of $\Coh(\calO_K^\syn)$ generated by all mod $p^n$-reductions of objects in $\Coh^{\refl}(\calO_K^\syn)$.
\end{dfn}

\begin{dfn}\label{mod p^n crystalline BK}
    Let $BK(\frakS/p^n)^{G_K}_{\calN}$ denote the category of finite free Breuil--Kisin modules $\frakM$ over $\frakS/p^n$ with $G_K$-actions on $\frakM_{\inf}$ and Nygaardian filtrations. Here Nygaardian filtrations mean filtrations $\calN^\bullet\varphi^*\frakM$ on $\varphi^*\frakM$ such that $\Phi(\calN^\bullet\varphi^*\frakM)\subseteq E^\bullet\frakM$ (hence not necessarily saturated).
\end{dfn}

Finally, we can prove the second main theorem in this section.
\begin{thm}

The natural functor $\Coh^{\refl}(\calO_K^\syn)/p^n\to BK(\frakS/p^n)^{G_K}_{\calN}$ is fully faithful. In particular,
   for $F_1,F_2\in \Coh^{\refl}(\calO_K^\syn)$, there is an isomorphism $F_1/p^n\simeq F_2/p^n$ if and only if there is an isomorphism of the mod $p^n$-reductions of the associated Breuil--Kisin modules with $G_K$-actions and Nygaard filtrations.
\end{thm}

\begin{proof}
    This directly follows from Theorem \ref{key1} and Proposition \ref{key2}.
\end{proof}

   Combining Theorems \ref{main} and \ref{main2}, we then get a ``practical" way to study congruences of first syntomic cohomology groups.

   Before we move into the next section, we record the following result explaining how congruences of reflexive $F$-gauges impact Hodge--Tate weights. We first recall how to read off Hodge--Tate weights from the Nygaard filtrations and the conjugate filtrations.

   \begin{dfn}[{\cite[Notation 1.1(c) and Lemma 2.4]{gao2024integral}}]\label{conjugate fil}
    For any Breuil--Kisin module $\frakM$ attached to a crystalline $\bZ_p$-representation $T$ with non-negative Hodge--Tate weights $0\leq r_1\leq r_2\cdots \leq r_d$, the $i$-th conjugate filtration on $\frakM_{\HT}:=\frakM/E$ is defined as the image of the injective map
    \[\calN^i\varphi^*\frakM/\calN^{i+1}\varphi^*\frakM\xhookrightarrow{\times E^{-i}} \frakM_{\HT}.
    \]
    The conjugate filtration $\Fil_{\bullet}\frakM_{\HT}$ is an increasing filtration. Write $\gr_i\frakM_{\HT}:=\Fil_{i}\frakM_{\HT}/\Fil_{i-1}\frakM_{\HT}$. For any integer $i$, the multiplicity $m(i)$ of the integer $i$ appearing in the multi-set $\HT(T)$ of the Hodge--Tate weights is $\dim_K\gr_i\frakM_{\HT}[1/p]$.
\end{dfn}

 \begin{prop}\label{HT weights}
     Let $M_1,M_2$ be two reflexive $F$-gauges over $\calO_K$. If $M_1/p^n\cong M_2/p^n$ for some $n$, then the multi-sets of Hodge--Tate weights $\HT(M_1), \HT(M_2)$ are the same.
 \end{prop}

 \begin{proof}
     Let $\frakM_1,\frakM_2$ be the associated Breuil--Kisin modules. We may assume both are effective.

Write $A^j_i:=\gr^j_\calN\varphi^*\frakM_i$ for $i=1,2$. Each $A_i^j$ is finite free over $\calO_K$. Then multiplying by $E$ induces an injective map $A^{j-1}_i\xhookrightarrow{\times E}A^{j}_i$. Let $m_i(j)$ be the multiplicity of the integer $j$ appearing in $\HT(M_i)$. Then
\[
m_i(j)=\dim_K\frac{A_i^j[1/p]}{A_{i}^{j-1}[1/p]}=\dim_K{A_i^j[1/p]}-\dim_K{A_i^{j-1}[1/p]}.
\]
     
    The congruence $M_1/p^n\cong M_2/p^n$ implies $\calN^\bullet\varphi^*\frakM_1/p^n\cong \calN^\bullet\varphi^*\frakM_2/p^n$. In particular, there is a compatible family of isomorphisms $A_1^j/p^n\cong A_2^j/p^n$. 

     Note that ${\rm Rank}_{\calO_K/p^n}A_i^j/p^n=\dim_KA_i^j[1/p]$. Then we must have $\dim_KA_1^j[1/p]=\dim_KA_2^j[1/p]$. Hence $\HT(M_1)=\HT(M_2)$.
 \end{proof}

\section{Comparison of extension groups}\label{extension}
This section is devoted to proving Theorem \ref{cohernel boundedness}.

We first consider the category $\Vect^{\varphi}((\calO_K)_{\Prism},\calO_{\Prism})$ of prismatic $F$-crystals. Note that $\Vect^{\varphi}((\calO_K)_{\Prism},\calO_{\Prism})$ is an exact category. So it is legal to consider the extension group $\Ext^1(\calO_{\Prism},\bM)$ for any $\bM\in \Vect^{\varphi}((\calO_K)_{\Prism},\calO_{\Prism})$. Let $T$ be the \'etale realization of $\bM$, whose Hodge--Tate weights are $r_1\leq r_2\cdots\leq r_d$. As the \'etale realization functor is exact, we then get a natural group homomorphism
\[
\Ext^1(\calO_{\Prism},\bM)\to H^1_f(K,T).
\]

\begin{prop}\label{crystal-rep}
    The homomorphism $\Ext^1(\calO_{\Prism},\bM)\to H^1_f(K,T)$ is injective whose cokernel is killed by $p^{s_1}$ for some $s_1$ which only depends on the ramification degree $e$ and $r_1, r_d$.
\end{prop}
\begin{proof}
    Consider an extension $0\to \bM\to \tilde\bM\to \calO_{\Prism}\to 0$, whose \'etale realization is the trivial extension $0\to T\to T\oplus \bZ_p\to \bZ_p\to 0$.  As $\Vect^{\varphi}((\calO_K)_{\Prism},\calO_{\Prism})\simeq \Rep^{\cris}_{\bZ_p}(G_K)$, we see that $\tilde\bM\cong \bM\oplus \calO_{\Prism}$. And the extension is just the trivial extension $0\to \bM\to \bM\oplus\calO_{\Prism}\to \calO_{\Prism}\to 0$.

    Now we prove the statement about the cokernel $C:=\Coker(\Ext^1(\calO_{\Prism},\bM)\to H^1_f(K,T))$. Choose any $x\in H^1_f(K,T))$, which corresponds to an extension $0\to T\xrightarrow{f} \tilde T\xrightarrow{g} \bZ_p\to 0$. Then we get a left exact sequence of Breuil--Kisin modules
    \[
    0\to \frakM\xrightarrow{\alpha} \tilde\frakM\xrightarrow{\beta} \frakS
    \]
     and $\Coker(\beta)$ is killed by $(p,u)^k$ for some positive integer $k$ by  \cite[Proposition 3.4]{Coupek26}.

     Let $J$ be the image of $\beta$. Then $J$ is an ideal of $\frakS$, which satisfies the conditions in \cite[Situation 2.1]{gabber2025bounding}:
     \begin{enumerate}
         \item The ideal $J$ is cofinite, i.e. $(p,u)^N\subseteq J$ for some $N$;
         \item $E^j\cdot J\subseteq \varphi(J)\cdot \frakS$ for some $j\in \bN$.
     \end{enumerate}
     The condition $(2)$ for $J$ follows from that $\varphi:\frakS\to \frakS$ is faithfully flat. In fact, we can choose $j=\max\{0,r_d\}$ for $J$.

     So by \cite[Propositions 2.15 and 2.16]{gabber2025bounding}, there exists an integer $s_1(e,r_1,r_d)$ only depending on $e,r_1,r_d$ such that $(p^{s_1})\subseteq J$. 

     Now consider the extension class $p^{s_1}x\in H^1_f(K,T)$, which then corresponds to the extension \[
     0\to T\xrightarrow{f_{s_1}} \tilde T^{s_1}\xrightarrow{g_{s_1}} \bZ_p\to 0
     \]
     where $\tilde T^{s_1}:=g^{-1}(p^{s_1}\bZ_p)$ and $f_{s_1}=f$, $g_{s_1}=g/p^{s_1}$.

     We can consider similar constructions for Breuil--Kisin modules. Define $\tilde\frakM^{s_1}:=\beta^{-1}(p^{s_1}\frakS)$ and there is a short sequence
     \[
     0\to \frakM\xrightarrow[]{\alpha}\tilde\frakM^{s_1}\xrightarrow[]{\beta/p^{s_1}} \frakS\to 0.
     \]
     It is easy to see that this sequence is left exact. To see it is also right exact, choose any $x\in \frakS$. We know that $p^{s_1}\frakS\subseteq Im(\beta)$, which means there exists $y\in \tilde\frakM$ such that $\beta(y)=p^{s_1}x$. Hence $y\in \tilde\frakM^{s_1}$ and $\beta(y)/p^{s_1}=x$. 

     It is also clear that $\tilde\frakM^{s_1}$ is a finite free Breuil--Kisin module with inherited $G_K$ action on $\tilde\frakM^{s_1}\otimes_{\frakS}\Ainf$. We want to show it is indeed a crystalline Breuil--Kisin module by checking conditions in Definition \ref{crystalline BK}. The condition $(1)$ is clear as $\tilde\frakM^{s_1}\subseteq \tilde\frakM$ which is stable under the action of $G_{\infty}$.

     For condition $(2)$, note that both $\tilde\frakM^{s_1}$ and $\tilde \frakM$ are finite free and $\tilde\frakM^{s_1}[1/p]=\tilde\frakM[1/p]$. So we get $\tilde\frakM^{s_1}/u\subseteq \tilde\frakM/u\subseteq \tilde\frakM^{s_1}/u[1/p]=\tilde\frakM/u[1/p]$. Hence $\tilde\frakM^{s_1}/u$ is also $G_K$-stable.

     For condition $(3)$, this actually directly follows from that both $\tilde\frakM$ and $\frakS$ are crystalline.

     So $ 0\to \frakM\xrightarrow[]{\alpha}\tilde\frakM^{s_1}\xrightarrow[]{\beta/p^{s_1}} \frakS\to 0$ is indeed a short exact sequence of crystalline Breuil--Kisin modules, which corresponds to a short exact sequence of prismatic $F$-crystals. Hence the cokernel of $\Ext^1(\calO_{\Prism},\bM)\to H^1_f(K,T)$ is killed by $p^{s_1}$.
\end{proof}

The exponent $s_1$ can be made explicit as in \cite{gabber2025bounding}. Let $j=\max\{0,r_d\}$ and $l=\max\{-r_1,0\}$. Also let $\sigma=\lfloor\frac{ej}{p-1}\rfloor$ and $\rho=d(e,j)$ where
\[
d(e,j):= \min \left( \sum_{i=1}^{n-1} \lfloor \frac{p^i}{e} \rfloor + \lfloor \frac{\lfloor\frac{e \cdot j}{p-1}\rfloor+1}{e} \rfloor + nj, \left( \frac{\lfloor\frac{e \cdot j}{p-1}\rfloor}{p^{n-1}(p-1)} + n \right) \cdot j \right)
\]
and $n$ is the smallest natural number such that $ej<p^n(p-1)$. Then we can take
\[s_1=(\rho+\max\{j,l\})\cdot\sigma.\]

In particular, if $r_d\leq 0$, then $j=0$ and $\sigma=0$. In this case, we can take $s_1=0$. This recovers \cite[Theorem A]{Coupek26}.

Now we continue to investigate $F$-gauges.   Let $M\in\Coh^{\refl}(\calO_K^\syn)$ and $\bM$ its attached prismatic $F$-crystal. Let $T$ be its \'etale realization with Hodge--Tate weights $r_1\leq r_2\leq \cdots \leq r_d$.

\begin{lem}

There is an injective group homomorphism
\[
H^1(\calO_K^\syn,M)\to \Ext^1(\calO_{\Prism},\bM).
\]

\end{lem}

\begin{proof}
 For the ringed site $(\calO_K^\syn,O)$, we know that the syntomic cohomology group $H^1(\calO_K^\syn, M)$ is exactly the extension group of the trivial $F$-gauge $O$ by $M$.  For an extension $0\to M\to \tilde M\to O\to 0$, we claim that $\tilde M$ is also a  reflexive coherent sheaf. This can be checked by pulling back along the flat map $\calO_C^{\calN}\to\calO_K^{\syn}$ using \cite[Definition 6.6.4]{bhatt22}. Or another way to see this is to use the faithfully flat map $\widehat B\to \calO_K^{\syn}$ as in \cite[Remark 6.6.12]{bhatt22} where $\widehat B$ is a $3$-dimensional regular Noetherian domain. Any extension of  reflexive coherent $\widehat B$-modules is still a  reflexive coherent $\widehat B$-module.

Note that the restriction functor  $\Coh^{\refl}(\calO_K^\syn)\to \Vect^{\varphi}((\calO_K)_{\Prism},\calO_{\Prism})$ is also exact as $\calO_K^{\Prism}\to \calO_K^\syn$ is \'etale.
    So we get a natural group homomorphism $H^1(\calO_K^\syn, M)\to H^1_f(K,T)$. To show it is injective, consider an extension class in $H^1(\calO_K^\syn,M)$ corresponding to
    \[
    0\to M\to \tilde M\to O\to 0,
    \]
    which is sent to the split extension. Then $ 0\to M\to \tilde M\to O\to 0$ must be isomorphic to the trivial extension $0\to M\to M\oplus O\to O\to 0$ due to the equivalence $\Coh^{\refl}(\calO_K^\syn)\simeq \Vect^{\varphi}((\calO_K)_{\Prism},\calO_{\Prism})$.

\end{proof}

We are interested in the cokernel of the above injective map. Before investigating the general case, we first deal with the special case where $M$ is effective, equivalently $r_1\geq 0$.

\begin{prop}\label{effective case}
    When $M$ is effective, then the injective map $H^1(\calO_K^\syn,M)\to \Ext^1(\calO_{\Prism},\bM)$ is an isomorphism.
\end{prop}

\begin{proof}
    We just need to prove that for any extension $0\to \bM\xrightarrow{f} \tilde\bM\xrightarrow{g} \calO_{\Prism}\to 0$, the induced sequences of Nygaard filtrations on their associated Breuil--Kisin modules
    \[
    0\to \calN^{i}\varphi^*\frakM\xrightarrow{f} \calN^i\varphi^*\tilde\frakM\xrightarrow{g} \calN^i\varphi^*\frakS\to 0
    \]
    are exact.
    
    As $\bM$ is effective, so is $\tilde\bM$. So we only need to check for all $i\geq 0$. Firstly, we verify the left exactness. Consider the following commutative diagram
    \begin{equation*}
        \xymatrix{
        0\ar[r]& \varphi^*\frakM\ar[r]\ar[d]&\varphi^*\tilde\frakM\ar[r]\ar[d]&\varphi^*\frakS\ar[r]\ar[d]& 0\\
        0\ar[r]& \frakM\ar[r]&\tilde\frakM\ar[r]&\frakS\ar[r]& 0.
        }
    \end{equation*}
    Note that all vertical maps are injective. Recall that any $x\in \varphi^*\frakM$, we have $x\in \calN^i\varphi^*\frakM$ if and only if $\Phi(x)\in E^i\frakM$. So we get the induced commutative diagram
    
      \begin{equation*}
        \xymatrix{
        0\ar[r]& \varphi^*\frakM/{\calN^i\varphi^*\frakM}\ar[r]\ar[d]&\varphi^*\tilde\frakM/{\calN^i\varphi^*\tilde\frakM}\ar[r]\ar[d]&\varphi^*\frakS/{\calN^i\varphi^*\frakS}\ar[r]\ar[d]& 0\\
        0\ar[r]& \frakM/E^i\ar[r]&\tilde\frakM/E^i\ar[r]&\frakS/E^i\ar[r]& 0
        }
    \end{equation*}
    where all vertical maps are again injective. Hence $\calN^i\varphi^*\frakM=\varphi^*\frakM\cap \calN^i\varphi^*\tilde\frakM$, which implies $0\to \calN^i\varphi^*\frakM\to \calN^i\varphi^*\tilde\frakM\to \calN^i\varphi^*\frakS$ is left exact. 

    It remains to prove the surjectivity. Note that $\calN^i\varphi^*\frakS=E^i\varphi^*\frakS$. So any $y\in \calN^i\varphi^*\frakS$ must have a lift in $E^i\varphi^*\tilde\frakM$ by the surjectivity of the map $\varphi^*\tilde\frakM\to \varphi^*\frakS$. As $E^i\varphi^*\tilde\frakM$ is contained in $\calN^i\varphi^*\tilde\frakM$, the surjectivity follows.
    
\end{proof}

The injective map is not an isomorphism in general. But after inverting $p$, it is indeed an isomorphism.

 \begin{lem}\label{coker-Nya}
     Let $0\to \bM\xrightarrow{f}\tilde\bM\xrightarrow{g} \calO_{\Prism}\to 0$ be an extension of prismatic $F$-crystals with attached Breuil--Kisin module $\frakM,\tilde\frakM,\frakS$ respectively. Then the induced short sequences
     \[
     0\to (\calN^i\varphi^*\frakM)[1/p]\xrightarrow{f}(\calN^i\varphi^*\tilde\frakM)[1/p]\xrightarrow{g}(\calN^i\varphi^*\frakS)[1/p]\to 0
     \]
     are exact.
 \end{lem}
 \begin{proof}
     We may twist the three Breuil--Kisin modules by some integer $n$ such that all of them are effective:
     \[
    0\to \frakM\{-n\}\xrightarrow{f}\tilde\frakM\{-n\}\xrightarrow{g} \calO_{\Prism}\{-n\}\to 0.
     \]

     We can use the same argument as in Proposition \ref{effective case} to prove the left exactness. It remains to prove the right exactness.

     Note that for any effective Breuil--Kisin module $\frakL$ attached to a weakly admissible filtered $\varphi$-module $D$, there is a map 
     \[
     ev_p:\varphi^*\frakL\twoheadrightarrow \varphi^*\frakL/E\to D_K=D\otimes_{K_0}K
     \]
     such that $ev_p(\calN^i\varphi^*\frakL)[1/p]=\Fil^iD_K$. And there exists a short exact sequence
     \[
     0\to \calN^{i-1}\varphi^*\frakL\xrightarrow{\times E}\calN^i\varphi^*\frakL\to ev_p(\calN^i\varphi^*\frakL)\to 0
     \]
     which follows from $E\varphi^*\frakL\cap \calN^i\varphi^*\frakL=E\calN^{i-1}\varphi^*\frakL$.

Suppose we have a short exact sequence of effective prismatic $F$-crystals $0\to \bL_1\xrightarrow[]{f} \bL_2\xrightarrow[]{g} \bL_3\to 0$ which induces a short exact sequence of associated weakly admissible filtered $\varphi$-modules $0\to D_1\xrightarrow{f}D_2\xrightarrow[]{g}D_3\to 0$. In particular, there is a short exact sequence
     \[
     0\to \Fil^iD_{1,K}\to \Fil^i D_{2,K}\to \Fil^iD_{3,K}\to 0 
     \]

   Let $\frakL_1,\frakL_2,\frakL_3$ be their associated Breuil--Kisin modules. We want to prove the short exact sequence of Nygaard filtrations on the Breuil--Kisin modules are exact after inverting $p$. We prove this by induction. Note that for $i\leq 0$, we have $\calN^i\varphi^*\frakL_j=\varphi^*\frakL_j$ for $j=1,2,3$. So the exactness is automatic. 

    Suppose there is a short exact sequence
    \[
    0\to \calN^{r-1}\varphi^*\frakL_1[1/p]\xrightarrow{f}\calN^{r-1}\varphi^*\frakL_2[1/p]\xrightarrow{g}\calN^{r-1}\varphi^*\frakL_3[1/p]\to 0.
    \]
    We want to prove the exactness for the $r$-th Nygaard filtrations. We have the following commutative diagram
    \begin{equation*}
        \xymatrix{
        E\calN^{r-1}\varphi^*\frakL_1[1/p]\ar[r]\ar[d] & E\calN^{r-1}\varphi^*\frakL_2[1/p]\ar[r]\ar[d] & E\calN^{r-1}\varphi^*\frakL_3[1/p]\ar[d]\\
        \calN^{r}\varphi^*\frakL_1[1/p]\ar[r]\ar[d] & \calN^{r}\varphi^*\frakL_2[1/p]\ar[r]\ar[d] & \calN^{r}\varphi^*\frakL_3[1/p]\ar[d]\\
        \Fil^rD_{1,K}\ar[r] & \Fil^rD_{2,K}\ar[r] & \Fil^rD_{3,K}.
        }
    \end{equation*}

    For any $y\in \calN^r\varphi^*\frakL_3$, let $\bar y$ be its image in $\Fil^rD_{3,K}$. We can lift $\bar y$ to $\bar x\in \Fil^rD_{2,K}$. Then lift $\bar x$ to $x\in \calN^{r}\varphi^*\frakL_2[1/p]$. We then get $g(x)\in \calN^{r}\varphi^*\frakL_3[1/p]$. By commutativity, $g(x)$ maps to $\bar y$. Hence $g(x)-y=Ez\in E\calN^{r-1}\varphi^*\frakL_3[1/p]$ for some $z\in \calN^{r-1}\varphi^*\frakL_3[1/p]$. By our induction hypothesis, we can lift $z$ to $w\in \calN^{r-1}\varphi^*\frakL_2[1/p]$.

    Now define $x':=x-Ew\in \calN^{r}\varphi^*\frakL_2[1/p]$. Then $g(x')=g(x)-Eg(w)=g(x)-Ez=y$, which means $\calN^{r}\varphi^*\frakL_2[1/p]\to \calN^{r}\varphi^*\frakL_3[1/p]$ is surjective. So we are done.
 \end{proof}

 \begin{cor}
     The map $H^1(\calO_K^\syn,M)[1/p]\to \Ext^1(\calO_{\Prism},\bM)[1/p]$ is an isomorphism.
 \end{cor}
\begin{proof}
    This directly follows from Lemma \ref{coker-Nya} and the same argument concerning extension classes as in Proposition \ref{crystal-rep}.
\end{proof}

Our next goal is to bound the cokernel of the map $H^1(\calO_K^\syn,M)\to \Ext^1(\calO_{\Prism},\bM)$. To this end, we must analyze the Nygaard filtration carefully. Our main tool is to use the conjugate filtration in Definition \ref{conjugate fil}, especially the torsion $(\gr_i\frakM_{\HT})_{\tor}$ of the graded pieces of the conjugate filtration.

In the unramified case, Gao and Liu proved the following uniform boundedness.
\begin{thm}[{\cite[Theorem 7.4]{gao2024integral}}]
    When $K$ is absolutely unramified, uniformly for all $i$, the torsion part $(\gr_i\frakM_{\HT})_{\tor}$ is killed by $(r_d-1)!$.
\end{thm}

We record a similar result in the ramified case. 
\begin{thm}\label{ramified bound}
Let $a=E'(\pi)$. Assume $\frakM$ is effective. The torsion submodule $(\gr_n\frakM_{\HT})_{\tor}$ is killed by $c_n=\prod_{r_i\leq n-1}a(n-r_i)$. Consequently, it is killed by $p^{t_n}$ where 
\[
t_n=\lceil D_{n-1}\cdot v_p(a)+v_p(\prod_{r_i\leq n-1}(n-r_i))\rceil
\]
and $D_{n-1}=\#\{r_i\mid r_i\leq n-1\}$ which is the $\calO_K$-rank of $\Fil_{n-1}\frakM_{\HT}$.
\end{thm}

\begin{proof}
    Let $f=\Theta-an:\Fil_n\frakM_{\HT}\to \Fil_{n-1}\frakM_{\HT}$ be the $1$-degree shrinking operator in \cite{gao2024integral}. We denote its restriction to the submodule $\Fil_{n-1}\frakM_{\HT}$ by $g:\Fil_{n-1}\frakM_{\HT}\to \Fil_{n-1}\frakM_{HT}$.

    By \cite[Corollary 5.3]{gao2024integral}, the eigenvalues of the operator $g$ on $\Fil_{n-1}\frakM_{\HT}$ are exactly $a(r_i-n)$ for $r_i\leq n-1$. Hence the characteristic polynomial of $g$ on $\Fil_{n-1}\frakM_{\HT}$ is 
    \[
    P(X)=\prod_{r_i\leq n-1}(X-(a(r_i-n))).
    \]

    By the Cayley--Hamilton theorem, we see $P(g)=0$ as an endomorphism on $\Fil_{n-1}\frakM_{\HT}$. Write $P(X)=X\cdot Q(X)+P(0)$.

    Let $\bar x\in \gr_n\frakM_{\HT}$ be a non-zero torsion element. Then there exists a non-zero $\lambda\in \calO_K$ such that $\lambda \bar x=0$. Lift $\bar x$ to an element $x\in \Fil_n\frakM_{\HT}$. Then $\lambda x\in \Fil_{n-1}\frakM_{\HT}$. We then have
    \[
   P(g)(\lambda x)= g(Q(g)(\lambda x))+P(0)\lambda x=0.
    \]
Note that $g(\lambda x)=f(\lambda x)=\lambda f(x)$. Hence
\[
g(Q(g)(\lambda x))=\lambda f(Q(f)(x)).
\]
    
    Then $\lambda (f(Q(f)(x))+P(0)x)=0$. As $\frakM_{\HT}$ is finite free over $\calO_K$, we see $f(Q(f)(x))+P(0)x=0$.

    As $f(Q(f)(x))\in \Fil_{n-1}\frakM_{\HT}$, we see $P(0)x\in \Fil_{n-1}\frakM_{\HT}$, which implies $P(0)\bar x=0\in \gr_n\frakM_{\HT}$. Write $c_n=\prod_{r_i\leq n-1}a(n-r_i)$. Then $c_n=P(0)$ or $-P(0)$. Hence $c_n(\gr_n\frakM_{\HT})_{\tor}=0$.
\end{proof}

Now we can use the relation between Nygaard filtrations and conjugate filtrations to bound the cokernel of the map $H^1(\calO_K^\syn,M)\to \Ext^1(\calO_{\Prism},\bM)$.

\begin{thm}
    The cokernel of the map $H^1(\calO_K^\syn,M)\to \Ext^1(\calO_{\Prism},\bM)$ is killed by $p^{s}$ for some $s$, which only depends on the ramification degree $e$ and the Hodge--Tate weights.
\end{thm}

\begin{proof}
    Consider a short exact sequence of prismatic $F$-crystals
    \[
    0\to \bM_1\xrightarrow{f}\bM_2\xrightarrow{g}\bM_3\to 0
    \]
    where $\bM_3=\calO_{\Prism}$. 

    We need to prove that there exists an integer $s$ which only depends on the ramification degree $e$ and the Hodge--Tate weights such that for the left exact sequence
    \[
    0\to \calN^r\varphi^*\frakM_1\xrightarrow{f}\calN^r\varphi^*\frakM_2\xrightarrow{g}\calN^r\varphi^*\frakM_3
    \]
    the cokernel $\calN^r\varphi^*\frakM_3/\calN^r\varphi^*\frakM_2$ is killed by $p^s$ for all $r$.

    We may twist the sequence by the smallest non-negative integer $n$ such that $\frakM_1\{-n\}, \frakM_2\{-n\}, \frakM_3\{-n\}$ are all effective. Let $\Phi_i$ denote the map $\varphi^*\frakM_i\{-n\}\to \frakM_i\{-n\}$ for $i=1,2,3$.

    Choose a basis $e_3$ of $\frakM_3\{-n\}$, i.e. $\frakM_3\{-n\}\cong \frakS\cdot e_3$ with $\Phi_3(1\otimes e_3)=E^ne_3$. As $\frakS$-modules, we have $\frakM_2\{-n\}\cong \frakM_1\{-n\}\oplus \frakS\cdot e_3$ with $\Phi_2(1\otimes e_3)=E^ne_3+w$ for some $w\in \frakM_1\{-n\}$.

    Let $y=a(1\otimes e_3)\in \calN^{r+n}(\varphi^*\frakM_3\{-n\})$ be arbitrary. Then $a\in E^{\max\{r,0\}}\frakS$. The existence of a lift of $y$ in $\calN^{r+n}(\varphi^*\frakM_2\{-n\})$ is equivalent to the existence of an element $x\in \varphi^*\frakM_1\{-n\}$ such that
    \[
    \Phi_1(x)+aw\in E^{r+n}\frakM_1\{-n\}.
    \]

    Consider $\Omega_n:=\frakM_1\{-n\}/(Im(\Phi_1)+E^n\frakM_1\{-n\})$. The obstruction of lifting $y$ then reduces to the image $\bar w$ of $w$ in $\Omega_n$. By Lemma \ref{coker-Nya}, we know that for $a=1$, there always exists $x'\in \varphi^*\frakM_1\{-n\}[1/p]$ such that
    \[
    \Phi_1(x')+w\in E^{n}\frakM_1\{-n\}[1/p].
    \]
    In particular, this implies $\bar w=0\in \Omega_n[1/p]$.

Suppose $p^sw=\Phi_1(X)+E^nZ$ for some $X,Z$. Then multiplying by $a\in E^{\max\{r,0\}}\frakS$, we get
    \[
    p^s(aw)=\Phi_1(aX)+aE^nZ\in Im(\Phi_1)+E^{r+n}\frakM_1\{-n\}
    \]
    which means $p^s(\calN^r\varphi^*\frakM_3/\calN^r\varphi^*\frakM_2)=0$. So we need to find this exponent $s$ for $w$.

    We filter $\Omega_n$ by the descending sequence for $k=0,1,\cdots, n$
    \[
    F^k=\frac{Im(\Phi_1)+E^k\frakM_1\{-n\}}{Im(\Phi_1)+E^n\frakM_1\{-n\}}\subseteq \Omega_n.
    \]

    Note that $F^n=0$. For any $0\leq k\leq n-1$, we have
    \[
    \frac{F^k}{F^{k+1}}\cong \frac{E^k\frakM_1\{-n\}}{E^{k+1}\frakM_1\{-n\}+(Im(\Phi_1)\cap E^k\frakM_1\{-n\}}.
    \]

    Now we consider the natural surjective map $\psi:\frakM_{\HT}:=\frakM_1\{-n\}/E\twoheadrightarrow F^k/F^{k+1}$ defined as follows
    \[
    \psi(x\mod E):=E^kx \mod F^{k+1}.
    \]

For any $x\in \Ker(\psi)$, we have $E^kx=E^{k+1}y+\Phi_1(Z)$ for some $y\in \frakM_1\{-n\}$. Note that $\Phi_1(Z)\in E^k\frakM_1\{-n\}$, which means $Z\in \calN^k(\varphi^*\frakM_1\{-n\})$. Dividing by $E^k$, we get
\[
x\equiv x-Ey\equiv E^{-k}\Phi_1(Z) \mod E.
\]
By the definition of the conjugate filtration, we then see $\Ker(\psi)=\Fil_k\frakM_{1,\HT}$. Therefore we get
\[
\frac{F^k}{F^{k+1}}\cong \frac{\frakM_{HT}}{\Fil_k\frakM_{1,\HT}}
\]
for any $0\leq k\leq n-1$.

Define $Q_k=\frac{\frakM_{HT}}{\Fil_k\frakM_{\HT}}$ for all $k$. Note that we have short exact sequences
\[
0\to \gr_k\frakM_{\HT}\to Q_{k-1}\to Q_k\to 0.
\]
By Theorem \ref{ramified bound}, the $p$-torsion exponent of $\gr_k\frakM_{\HT}$ is bounded by $t_k$. Hence the $p$-torsion of $Q_k$ is bounded by $\sum_{i=k+1}^{r_d-r_1}t_i$, as for $k\geq r_d-r_1$ we have $Q_k=0$.

Finally, we can obtain
\[
s\leq \sum_{k=0}^{n-1}\sum_{i=k+1}^{r_d-r_1}t_i=\sum_{i=1}^{r_d-r_1}\min\{i,n\}\cdot t_i.
\]
Note that $v_p(a)=v_p(E'(\pi))\leq v_p(e)+\frac{e-1}{e}$. Hence each $t_i$ is also bounded by some constant only depending on the $e, r_1, r_2, \cdots, r_d$.
    
\end{proof}

\section{Global comparison}

Although prismatic $F$-gauges capture richer information than Galois representations in the torsion case, we still need to get back to Galois representations when it comes to global applications. In this section, we will discuss how the congruences of prismatic $F$-gauges can induce comparisons of finite local Bloch--Kato Selmer groups. With additional control at ramified primes $l\neq p$, this will also give comparisons of finite global Blcoh--Kato Slemer groups.

The tricky part is to verify congruences of prismatic $F$-gauges. Although Theorem \ref{main2} provides a criterion involving Breuil--Kisin modules, it is still difficult to verify it in practice. However, there are interesting cases where the congruences of prismatic $F$-gauges can be easily checked, namely the tensorial constructions of Barsotti--Tate representations. For abelian varieties, the local ($l=p$) Bloch--Kato Selmer group can be defined using the attached finite flat group schemes and fppf cohomology, hence prismatic $F$-gauges are not needed. But if we want to deal with tensorial constructions of their Tate modules, such as large symmetric powers, their finite flat group schemes are not suitable anymore while it is still meaningful to take the same tensorial constructions of prismatic $F$-gauges.

\subsection{The local comparison}
We first present a lemma that converts the bounded cokernel in Theorem \ref{cohernel boundedness} to a two sided comparison.

\begin{lem}\label{bounded-comparison}
Suppose there is a commutative diagram of $\bZ_p/p^n$-modules
\begin{equation*}
    \xymatrix{
    A_1\ar[r]^{\rho_1}\ar[d]^{\simeq}&B_1\ar@{^{(}->}[r]& C_1\ar[d]^{\alpha_n,\simeq}\\
    A_2\ar[r]^{\rho_2}&B_2\ar@{^{(}->}[r]& C_2
    }
\end{equation*}
such that the cokernels of $\rho_1,\rho_2$ are both killed by $p^c$. Then there are natural inclusions
\[
p^c\alpha_n(B_1)\subseteq B_2, \qquad p^c\alpha_n^{-1}(B_2)\subseteq B_1.
\]
\end{lem}
\begin{proof}
    This is obvious.
\end{proof}

We have the following direct consequence.
\begin{thm}\label{local-inclusion}
Let $N_1,N_2$ be reflexive $F$-gauges over $\calO_K^{\syn}$ and put $S_i=T_{\acute et}(N_i)$.  Suppose the Hodge--Tate weights of  $S_i[1/p]$ are $r_1\leq r_2\leq \cdots \leq r_d$.  Let
$c=c(e(K),r_1,r_2,\cdots, r_d)$ be the exponent for the composite injection in Theorem \ref{cohernel boundedness} (i.e. $s_1+s_2$ in Theorem \ref{cohernel boundedness}). Assume that $
 \gamma_{2n}:N_1/p^{2n}\xrightarrow{\sim}N_2/p^{2n}$ for some $n\geq a(e(K),r_1)$ as in Theorem \ref{main} and the \'{e}tale realization of $\gamma_{2n}$ reduces modulo $p^n$ to
$\alpha_n:S_1/p^n\simeq S_2/p^n$, with induced map $\alpha_{n,*}: H^1(K,S_1/p^n)\xrightarrow{\sim} H^1(K,S_2/p^n)$. Then we have natural inclusions
\[
 p^c\alpha_{n,*}H^1_f(K,S_1/p^n)\subseteq H^1_f(K,S_2/p^n),
 \qquad
 p^c\alpha_{n,*}^{-1}H^1_f(K,S_2/p^n)\subseteq H^1_f(K,S_1/p^n).
\]
\end{thm}

\begin{proof}
By Theorem \ref{cohernel boundedness} and \ref{main},  we just need to apply Lemma \ref{bounded-comparison} to
\[
 A_i=H^1(\calO_K^{\syn},N_i)/p^n,
 \qquad B_i=H_f^1(K,S_i/p^n), \qquad C_i=H^1(K,S_i/p^n).
\]
Note that we have natural maps
\[
H^1(\calO_K^{\syn},N_i)\to H^1_f(K,S_i)\to H^1(K,S_i).
\]
\end{proof}

\subsection{Global Bloch--Kato Selmer groups and bounded comparison}

Let $F$ be a number field and $p$ be a prime.  For a
finite free $\bZ_p$-representation $T$ of $G_{F}$, put
$V=T[1/p]$. Let $\Sigma$ be a finite set of places containing the archimedean places, all places above $p$ and the places where $T$ is ramified.  At a finite place $v$, recall the definitions of local Bloch--Kato Selmer groups
\begin{equation*}
 H_f^1(F_v,V)=
 \begin{cases}
 \Ker(H^1(F_v,V)\to H^1(F_v,V\otimes_{\bQ_p}B_{\cris})),&v\mid p,\\[3pt]
 \Ker(H^1(F_v,V)\to H^1(I_v,V)),&v\nmid p.
 \end{cases}
\end{equation*}
The integral and finite Bloch--Kato Selmer groups are defined as
\begin{align}
 H_f^1(F_v,T)
 &=\Ker(H^1(F_v,T)\to H^1(F_v,V)/H_f^1(F_v,V)),\\
 H_f^1(F_v,T/p^n)
 &=\Ima(H_f^1(F_v,T)/p^n\to H^1(F_v,T/p^n)).
 \label{finite-local-BK}
\end{align}
Note that  \eqref{finite-local-BK} is not defined using the unramified condition when $v\nmid p$.  This distinction is essential at a ramified place $v\nmid p$.

The corresponding global Bloch--Kato Selmer groups are
\begin{align}
 H_f^1(F,V) &=\Ker(H^1(G_{F,\Sigma},V)\to \bigoplus_{v\in\Sigma_{\rm fin}} \frac{H^1(F_v,V)}{H_f^1(F_v,V)}),\label{rational-global-BK-general}\\
 H_f^1(F,T)
 &=\Ker(H^1(G_{F,\Sigma},T)\to
 \bigoplus_{v\in\Sigma_{\rm fin}}
 \frac{H^1(F_v,T)}{H_f^1(F_v,T)}),
 \label{integral-global-BK-general}\\
 H_f^1(F,T/p^n)
 &=\Ker(H^1(G_{F,\Sigma},T/p^n)\to
 \bigoplus_{v\in\Sigma_{\rm fin}}
 \frac{H^1(F_v,T/p^n)}{H_f^1(F_v,T/p^n)}).
 \label{finite-global-BK-general}
\end{align}

\begin{dfn}\label{local-conditions-bounded}
Suppose that $\alpha_n:T_1/p^n\xrightarrow{\sim}T_2/p^n$.
At a finite place $v$, we say their local Bloch--Kato Selmer groups are bounded with exponent $d_v$ if
\begin{equation*}
 p^{d_v}\alpha_{n,*}H_f^1(F_v,T_1/p^n)
 \subseteq H_f^1(F_v,T_2/p^n),
 \qquad
 p^{d_v}\alpha_{n,*}^{-1}H_f^1(F_v,T_2/p^n)
 \subseteq H_f^1(F_v,T_1/p^n).
\end{equation*}
\end{dfn}

The generalization of Theorem \ref{local-inclusion} to the global setting is as follows.
\begin{thm}\label{global-bound}
Let $T_1,T_2$ be finite free $\bZ_p$-representations of
$G_{F,\Sigma}$, crystalline at every $v\mid p$ and suppose that
there is a global congruence $\alpha_{2n}:T_1/p^{2n}\xrightarrow{\sim}T_2/p^{2n}$.
For each $v\mid p$, assume that $T_i|_{G_{F_v}}$ is the
\'{e}tale realization of a reflexive $F$-gauge
$N_{i,v}$ and that there is a compatible congruence $\gamma_{v,2n}:N_{1,v}/p^{2n}\xrightarrow{\sim}N_{2,v}/p^{2n}$
whose \'{e}tale realization is $\alpha_{2n}|_{G_{F_v}}$.  Assume that $n\geq a(e(F_v),r_{1,v})$ as in Theorem \ref{main} where $r_{1,v}$ is the smallest Hodge--Tate weight and let $c_v$ be the exponent for the composite injection in Theorem \ref{cohernel boundedness} for the two local reflexive $F$-gauges (i.e. $s_1+s_2$ in Theorem \ref{cohernel boundedness}).

For $v\nmid p$, let $b_v(T_i)$ be the smallest exponent killing $H^1(I_v,T_i)_{\tor}$.  Let $\alpha_n$ be the reduction
of $\alpha_{2n}$ modulo $p^n$ and put
\[
 D=\max\{
 \max_{v\mid p}c_v,\ 
 \max_{\substack{v\in\Sigma_{\rm fin}\\v\nmid p}}
 (b_v(T_1)+b_v(T_2))
 \}.
\]

Then we have natural inclusions
\[
 p^{D}\alpha_{n,*}H_f^1(F,T_1/p^n)\subseteq H_f^1(F,T_2/p^n),
 \qquad
 p^{D}\alpha_{n,*}^{-1}H_f^1(F,T_2/p^n)\subseteq H_f^1(F,T_1/p^n).
\]

\end{thm}

\begin{proof}
For $v\mid p$, Theorem \ref{local-inclusion}
shows that the local Bloch--Kato Selmer groups are bounded with
exponent $c_v$.  For $v\nmid p$, the following Proposition \ref{l} gives exponent $b_v(T_1)+b_v(T_2)$. Then this theorem follows directly from the definition of global Bloch--Kato Selmer group.

\end{proof}

We now explain the finite exponents used at $v\nmid p$.  Define
\[
 H_{\rm{ur}}^1(F_v,T):=\Ker(H^1(F_v,T)\to H^1(I_v,T))
\]
and take $b_v(T)$ to be the smallest nonnegative integer such that $p^{b_v(T)}H^1(I_v,T)_{\tor}=0$. Thus $b_v(T)=0$ when there is no torsion.

\begin{prop}\label{l}
\begin{enumerate}
    \item For $v\nmid p$, the torsion subgroup $H^1(I_v,T)_{\tor}$ is finite and the restriction induces an injection
\begin{equation}\label{inj-tor}
 \frac{H_f^1(F_v,T)}{H_{\rm{ur}}^1(F_v,T)}\hookrightarrow H^1(I_v,T)_{\tor}.
\end{equation}
\item Still for $v\nmid p$, put $\calU_v^{(n)}(T)=\Ker(H^1(F_v,T/p^n)\to H^1(I_v,T/p^n))$. Then we have
\begin{equation}\label{unratified-inclu}
 p^{b_v(T)}H_f^1(F_v,T/p^n)\subseteq\calU_v^{(n)}(T),
 \qquad
 p^{b_v(T)}\calU_v^{(n)}(T)\subseteq H_f^1(F_v,T/p^n).
\end{equation}
Consequently, for an isomorphism
$\alpha_n:T_1/p^n\simeq T_2/p^n$, the bound $d_v$ in Definition \ref{local-conditions-bounded} can be taken as $b_v(T_1)+b_v(T_2)$. If $T_1, T_2$ are both unramified, or the inertia subgroup acts on both through a finite group of order prime to $p$, then one may take $d_v=0$.
\end{enumerate}

\end{prop}

\begin{proof}
\begin{enumerate}
    \item 
For $v\nmid p$, it is known that $H^1(I_v,T)$ is finitely generated. Since $H^1(I_v,T)[1/p]\xrightarrow{\sim}H^1(I_v,V)$
is an isomorphism, we see that $H_f^1(F_v,T)$ maps to $H^1(I_v,T)_{\tor}$.  By definition, we have $H_{\rm{ur}}^1(F_v,T)=\Ker(H^1(F_v,T)\to H^1(I_v,T))$, which then implies \eqref{inj-tor}.

\item
The first inclusion in \eqref{unratified-inclu} follows directly from the mod $p^n$ reduction of the inclusion
\[
 p^{b_v(T)}H_f^1(F_v,T)\subseteq H_{\rm{ur}}^1(F_v,T).
\]
More precisely, for any $x\in H_f^1(F_v,T/p^n)=H_f^1(F_v,T)/p^n$, we can choose a lift $\tilde x\in H_f^1(F_v,T)$. Then $p^{b_v(T)}\tilde x\in H_{\rm{ur}}^1(F_v,T)$. Note that there is a natural map $H_{\rm{ur}}^1(F_v,T)/p^n\to \calU_v^{(n)}(T)$. Hence the image of $p^{b_v(T)}\tilde x$ lies in $\calU_v^{(n)}(T)$ which implies $p^{b_v(T)}x$ lies in $\calU_v^{(n)}(T)$.

For the reverse inclusion, look at the short exact sequence
\[
 0\to T^{I_v}/p^n\to (T/p^n)^{I_v}\to H^1(I_v,T)[p^n]\to0.
\]
We see $p^{b_v(T)}(T/p^n)^{I_v}$ is contained in
$T^{I_v}/p^n$.  The inflation--restriction exact sequence identifies $\calU_v^{(n)}(T)$  with $H^1(\Gamma_v,(T/p^n)^{I_v})$ where $\Gamma_v=G_{F_v}/I_v\simeq\widehat{\bZ}$. Since $\Gamma_v$ has $p$-cohomological dimension $1$, we have
\[
 H^1(\Gamma_v,T^{I_v})/p^n\xrightarrow{\sim}H^1(\Gamma_v,T^{I_v}/p^n).
\]
Hence for any class $c\in \calU_v^{(n)}(T)$, the class $p^{b_v(T)}c$ living in $H^1(\Gamma_v,T^{I_v}/p^n)\cong H^1(\Gamma_v,T^{I_v})/p^n$  lifts to $H^1(\Gamma_v,T^{I_v})$, then inflates to a class in $H_{\rm{ur}}^1(F_v,T)$, and lies in the $H_f^1(F_v,T/p^n)$ after modulo $p^n$.  This proves the second inclusion.

The isomorphism $\alpha_n$ induces $\calU_v^{(n)}(T_1)\simeq \calU_v^{(n)}(T_2)$. Hence we get the bound $d_v=b_v(T_1)+b_v(T_2)$.  

Now if both $T_1,T_2$ are unramified, we can obviously take $d_v=0$. If the action of $I_v$ both factors through a finite group of order prime to $p$, then we claim that $T_i^{I_v}\to (T_i/p)^{I_v}$ is surjective. Indeed assume $I_v$-action factors through a finite group $H$. Then for any $x\in (T_i/p)^{I_v}$, we can take any lift $\tilde x\in T_i$ and define $x':=\frac{1}{\#H}\sum_{\gamma\in H}\gamma(\tilde x)$, which is a lift of $x$ in $T_i^{I_v}$. Then we see $H^1(I_v,T_i)[p]=0$ and we can take $d_v=0$.
\end{enumerate}
\end{proof}

\subsection{Lemmas on Galois cohomology}\label{Galois coh}
We record some results concerning Galois cohomology which will be useful in our final application. Let $F$ be a number field and $T$ be a finite free $\bZ_p$-representation of $G_F$. Put $V=T[1/p]$ and let $\Sigma$ be a finite set containing the archimedean places, the places above $p$ and every place where $T$ is ramified.  For each finite $v\in\Sigma$, put
\[
 Z_f^1(F_v,T)=
 \{z\in Z^1(F_v,T):[z]\in H_f^1(F_v,T)\},
 \qquad
 U_v^+(T)=[C^0(F_v,T)\xrightarrow d Z_f^1(F_v,T)].
\]
We  define
\begin{equation}\label{selmer complex}
 C_f(F,T,\Sigma)={\rm{Fib}}(C^\bullet(G_{F,\Sigma},T)\oplus\bigoplus_{v\in\Sigma_{\rm fin}}U_v^+(T)\to \bigoplus_{v\in\Sigma_{\rm{fin}}}C^\bullet(F_v,T)
 ),
\end{equation}
where the map is defined as the localization minus the inclusion of the local complexes. This mimics the Selmer complex in \cite{Nekovar06}.

The following proposition records the main properties of the source newform chosen in the Section \ref{global ex}.
\begin{prop}\label{source torsion exponent}
\begin{enumerate}
    \item If $H^0(F,T/p)=H_f^1(F,V)=0$, then we have $H^1(C_f(F,T,\Sigma))=H_f^1(F,T)=0$ and $H^2(C_f(F,T,\Sigma))$ is a finitely generated
$\bZ_p$-module.  In particular, the integer
\begin{equation*}
 b(F,T,\Sigma)=\min\{b\geq0:
 p^bH^2(C_f(F,T,\Sigma))_{\tor}=0\}
\end{equation*}
is finite.  
\item Under the assumptions in $(1)$, for every $r\geq1$, the global finite Selmer group  satisfies
\[
 p^{b(F,T,\Sigma)}H_f^1(F,T/p^r)=0.
\]
The same bound remains valid if $\Sigma$ is enlarged by adding primes at which $T$ is unramified.
\end{enumerate}
\end{prop}

\begin{proof}
\begin{enumerate}
    \item We first look at $H^0$ of \eqref{selmer complex}
\[
 H^0(F,T)\oplus\bigoplus_vH^0(F_v,T)
 \to\bigoplus_vH^0(F_v,T),
 \qquad
 (a,(b_v)_v)\longmapsto({\rm res}_v(a)-b_v)_v.
\]
It is surjective as a tuple $(c_v)_v$ is the image of
$(0,(-c_v)_v)$. Therefore $H^1(C_f(F,T,\Sigma))$ is isomorphic to  $H_f^1(F,T)$ by the long exact sequence associated with \eqref{selmer complex}.

The short exact sequence $0\to T\xrightarrow{\times p}T\to T/p\to0$ together with $H^0(F,T/p)=0$ shows that $H^1(G_{F,\Sigma},T)$ is $p$-torsion-free.  The Selmer group $H^1_f(F,T)$ is therefore torsion-free.  On the other hand, we have
\[
 H_f^1(F,T)[1/p]\simeq H_f^1(F,V)=0.
\]
Thus $H_f^1(F,T)=0$.  We know that the global Galois cohomology groups of $G_{F,\Sigma}$ and the local Galois cohomology groups are finitely generated, which then implies $H^2(C_f(F,T,\Sigma))$ is finitely generated. So $b(F,T,\Sigma)$ is a finite integer.
\item  Write $C_f=C_f(F,T,\Sigma)$. Then mod $p^r$  reduction gives the short exact sequence
\[
 0\to H^1(C_f)/p^r \to H^1(C_f\otimes^{\bL}\bZ/p^r)\to H^2(C_f)[p^r]\to0.
\]
  Since $U_v^+(T)$ is concentrated in degrees $0,1$, we know $H^2(U_v^+(T))=0$ and hence
\[
 H^1(U_v^+(T)\otimes^{\bL}\bZ/p^r)\simeq H_f^1(F_v,T)/p^r.
\]
Its image in $H^1(F_v,T/p^r)$ is precisely the local Bloch--Kato Selmer group.  So we get a surjective map 
\[H^1(C_f\otimes^{\bL}\bZ/p^r)\twoheadrightarrow H_f^1(F,T/p^r)=\Ker(H^1((C^\bullet(G_{F,\Sigma},T)\oplus\bigoplus_{v\in\Sigma_{\rm fin}}U_v^+(T))\otimes\bZ/p^r)\to H^1((\bigoplus_{v\in\Sigma_{\rm{fin}}}C^\bullet(F_v,T))\otimes\bZ/p^r)).\]  Since $H^1(C_f)=0$, the group $H_f^1(F,T/p^r)$ is then a quotient of $H^1(C_f\otimes^{\bL}\bZ/p^r)\cong H^2(C_f)[p^r]$ and is killed by $p^{b(F,T,\Sigma)}$.  Finally, enlarging $\Sigma$ by adding unramified primes does not change the Selmer group, so the same bound remains valid.
\end{enumerate}
\end{proof}

The following lemma will be used to reduce the impact of Steinberg primes in Section \ref{global ex}.
\begin{lem}\label{unramified H^1}
Let $v\nmid p$ and $q_v=\#k_v$ and let $M$ be a finite $\bZ/p^n$-module with an unramified $G_{F_v}$-action.  Let $\Frob_v$ denote the arithmetic Frobenius. Then there is a natural exact sequence
\[
 0\to\frac{M}{(\Frob_v-1)M}\to H^1(F_v,M)\to\Ker(\Frob_v-q_v:M\to M)\to0.
\]
In particular, if both $\Frob_v-1$ and $\Frob_v-q_v$ are
automorphisms of $M$, then $H^1(F_v,M)=0$.
\end{lem}

\begin{proof}
The inflation-restriction exact sequence gives
\[
 0\to H^1(G_{k_v},M)\to H^1(F_v,M)\to H^1(I_v,M)^{G_{k_v}}\to0,
\]
since $G_{k_v}\simeq\widehat{\bZ}$ has $p$-cohomological dimension $1$.  The first term $H^1(G_{k_v},M)$ is $M/(\Frob_v-1)M$. As $I_v$ acts trivially on $M$, we see $H^1(I_v,M)=\Hom_{\rm cont}(I_v,M)=\Hom_{\rm cont}(\bZ_p(1),M)$, where the second equality follows from that $M$ is a $\bZ/p^n$-module. Choose a topological generator $\tau$ of $\bZ_p(1)$. Note that the action of $\Frob_v$ on $H^1(I_v,M)$ is as follows: for any $c\in H^1(I_v,M)$, $(\Frob_v\cdot c)(\tau)=\Frob_v(c(\Frob_v^{-1}\cdot\tau\cdot \Frob_v))$. The relation $\Frob_v^{-1}\cdot \tau\cdot\Frob_v=\tau^{(q_v^{-1})}$ identifies $H^1(I_v,M)^{G_{k_v}}=\Ker(\Frob_v-q_v:M\to M)$. So we are done.
\end{proof}

The following lemma will be used to study the properties of the constructed newforms in the Section \ref{global ex}.

\begin{lem}\label{target-vanishing}
Suppose $H^0(F,T/p)=0$. If for some integers $N>B\geq0$ we have $p^B H_f^1(F,T/p^N)=0$, then we also have $H_f^1(F,T)=H_f^1(F,V)=0$.
\end{lem}

\begin{proof}
Write $H_\Sigma^1(T)=H^1(G_{F,\Sigma},T)$ for simplicity, which is a finitely generated $\bZ_p$-module.  The long exact sequence
associated with
\[
 0\to T\xrightarrow{\times p }T\to T/p\to0
\]
gives a sequence
\[
 H^0(F,T/p)\to H_\Sigma^1(T)\xrightarrow{\times p}H_\Sigma^1(T).
\]
Thus $H^0(F,T/p)=0$ implies that $H_\Sigma^1(T)$ has no $p$-torsion.  It is therefore finite free over $\bZ_p$. So the natural map $ H_\Sigma^1(T)\to H^1(G_{F,\Sigma},V)$ is injective. We may then regard
\[
 L:=H_f^1(F,T)=H_\Sigma^1(T)\cap H_f^1(F,V)\]
as an intersection inside $H^1(G_{F,\Sigma},V)$.

We claim $ L\cap p^N H_\Sigma^1(T)=p^N L$.  More precisely, if $px\in L$, then for every finite $v\in\Sigma$ the localization of $px$ belongs to $H_f^1(F_v,V)$.  Since multiplication by $p$ is invertible on $H_f^1(F_v,V)$, the localization of $x$ also lies in $H_f^1(F_v,V)$ and hence $x\in L$.  Iterating this
observation gives $L\cap p^N H_\Sigma^1(T)=p^N L$.

The exact sequence $0\to T\xrightarrow{p^N}T\to T/p^N\to0$ gives an injection
\[
 H_\Sigma^1(T)/p^N\hookrightarrow H^1(G_{F,\Sigma},T/p^N).
\]
Combining this with $ L\cap p^N H_\Sigma^1(T)=p^N L$
yields an injection
\[
 L/p^NL\hookrightarrow H^1(G_{F,\Sigma},T/p^N),
\]
which then factors through an injection
\[
L/p^NL\hookrightarrow H_f^1(F,T/p^N).
\]

Suppose that $L\neq0$.  Note that $L$ finite free over $\bZ_p$.  Choose $x\in L\setminus pL$.  The image of $x$ in $L/p^NL$ has exact order $p^N$ and by $ L/p^NL\hookrightarrow H_f^1(F,T/p^N)$ the same is true for its image in $H_f^1(F,T/p^N)$.  This contradicts $p^B H_f^1(F,T/p^N)=0$, as $B<N$.  Hence $L=H_f^1(F,T)=0$ which also implies $L[1/p]=H_f^1(F,V)=0$.

\end{proof}

\subsection{Tensorial constructions of Barsotti--Tate representations}
The difficulty of applying Theorem \ref{global-bound} lies in the verification of the congruences of reflexive $F$-gauges.  In this subsection, we want to focus on one specific case, i.e. tensorial constructions of Barsotti--Tate representations. Let us first recall the Bloch--Kato Selmer groups of abelian varieties.

Let $A$ be an abelian variety over $K$ with good reduction $\calA$ over $\calO_K$. The Kummer maps identify (cf. \cite[Section 3]{BK90})
\[
 H_f^1(K,T_p(A))=A(K)^\wedge_p,\qquad
 H_f^1(K,V_p(A))=A(K)^\wedge_p\otimes_{\bZ_p}\bQ_p,\qquad
 H_f^1(K,T_p(A)/p^n)=A(K)/p^nA(K)
\]
and then we get 
\[
H_{\rm{fppf}}^1(\calO_K,\calA[p^n])\simeq\calA(\calO_K)/p^n\calA(\calO_K)=A(K)/p^nA(K)\simeq H_f^1(K,T_p(A)/p^n),
\]
which means $H_f^1(K,T_p(A)/p^n)$ is determined by the finite flat group scheme $\calA[p^n]$.

Moreover, if $e(K/\bQ_p)<p-1$, then Raynaud \cite[Corollaire 3.3.6(1)]{Raynaud74} shows that an isomorphism
\[
 T_p(A_1)/p^n\simeq T_p(A_2)/p^n
\]
extends uniquely to $\calA_1[p^n]\simeq\calA_2[p^n]$. So in this small ramification case, the finite Bloch--Kato Selmer group $H_f^1(K,T_p(A)/p^n)$ is already determined by the Galois representation $T_p(A)/p^n$.

In many questions, we have to study symmetric powers of Barsotti--Tate representations or other tensorial constructions thereof. In these cases, finite flat group schemes will not be enough to determine the local Bloch--Kato Selmer groups of these tensorial constructions. However, we are free to form tensorial constructions of the attached reflexive $F$-gauges. And this will enable us to apply Theorem \ref{global-bound} to transfer arithmetic information through congruences.

In this specific setting, we can focus on the subcategory of prismatic $F$-gauges on $\calO_K^{\syn}$ which are vector bundles. Indeed, the tensor product of two reflexive $F$-gauges might not be reflexive anymore. Fortunately, the reflexive $F$-gauges attached to abelian schemes over $\calO_K$ are all vector bundles (cf. \cite[Theorem A]{gardner2024algebraicity}).

\begin{dfn}
By a tensorial construction, we mean tensor product, dual, $\Sym^r$ and $\bigwedge^q$ for every $r,q\geq0$. In particular, for  any finite set of vector bundles $\{M_i\}_{i\in I}$ on $\calO_K^{\syn}$ and any such tensorial construction $R$ which can be applied to $\{M_i\}_{i\in I}$, we have
\begin{enumerate}
\item $R(\{M_i\})$ is again a vector bundle,
\item $T_{\acute et}(R(\{M_i\}))\simeq R(\{T_{\acute et}(M_i)\})$,
\item $R(\{M_i\})/p^m\simeq R(\{M_i/p^m\})$ for every $m$.
\end{enumerate}

\end{dfn}

We then have the following corollary.

\begin{cor}\label{tensorial-BS}
Let $F$ be a number field and $p$ be an odd prime. Let  $\Sigma$ be a finite set of places containing the archimedean places, the places above $p$. Let $\{T_{1,i}\}_{i\in I},\{T_{2,i}\}_{i\in I}$ be two sets of $\bZ_p$-representations of $G_{F,\Sigma}$ indexed by the same finite set $I$ such that for any $t=1,2$ and $i\in I$, each $T_{t,i}$ is Barsotti--Tate at every $v\mid p$ with associated reflexive $F$-gauge $M_{t,i,v}$. Let $R$ be a tensorial construction and put $S_{t}=R(\{T_{t,i}\}_{i\in I})$.

Assume $e(F_v/\bQ_p)<p-1$ for each $v\mid p$. Let $n$ be an integer such that $n\geq a(e(F_v),r_{1,v})$ for all $v\mid p$ where $a(e(F_v),r_{1,v})$ is the exponent provided in Theorem \ref{main} and $r_{1,v}$ is the smallest Hodge--Tate weight of $S_t|_{G_{F_v}}$. And let $c_v$ be the exponent for the composite injection in Theorem \ref{cohernel boundedness} for $S_t|_{G_{F_v}}$.

Suppose that there is a congruence for each $i\in I$
\[
\alpha_{i,2n}:T_{1,i}/p^{2n}\xrightarrow{\sim}T_{2,i}/p^{2n}.
\]

Set
\[
 D=\max\{\max_{v\mid p}c_v,\ \max_{\substack{v\in\Sigma_{\rm fin}\\v\nmid p}} (b_v(S_1)+b_v(S_2))\}.
\]

Then the congruence induced by $\alpha_n: S_1/p^n\xrightarrow{\simeq}S_2/p^n$ gives two inclusions
\[
 p^{D}\alpha_{n}H_f^1(F,S_1/p^n)\subseteq H_f^1(F,S_2/p^n),
 \qquad
 p^{D}\alpha_{n}^{-1}H_f^1(F,S_2/p^n)\subseteq H_f^1(F,S_1/p^n).
\]

\end{cor}

\begin{proof}
For each $v\mid p$, we have the induced isomorphism
\[
M_{1,i,v}/p^{2n}\xrightarrow{\simeq}M_{2,i,v}/p^{2n}.
\]
Applying the tensorial construction $R$ then gives an isomorphism $R(M_{1,i,v})/p^{2n}\xrightarrow{\simeq}R(M_{2,i,v})/p^{2n}$ with \'{e}tale realization $R(\alpha_{i,2n})|_{G_{F_v}}$.  All hypotheses of Theorem \ref{global-bound} are therefore satisfied.
\end{proof}

\newpage
\section{Global example} \label{global ex}

The goal of this section is to present some interesting global examples. We will start with a newform about which we know enough arithmetic information. Then we use the deformation theory in \cite{CP18} to construct newforms which are congruent to the given newform. Based on Theorem \ref{global-bound} or Corollary \ref{tensorial-BS}, we can then show the vanishing of the Bloch--Kato Selmer group of a suitable tensorial construction of the Galois representations attached to these newforms and show they satisfy the rank part of the Bloch--Kato conjecture. These cases are indeed non-ordinary and do not lie in the usual small ramification-weight range (in particular, the Fontaine--Laffaille range).

\subsection{The source newform}

Let $f=43.2.\rm c.\rm b$ in LMFDB. This cuspidal newform belongs to $S_2(\Gamma_0(43),\epsilon)$, where $\epsilon$ is the primitive Dirichlet character of conductor $43$ and order $3$. The basic data of $f$ can be checked directly on the LMFDB page \cite{LMFDB43}.

The Hecke field $K_f$ is isomorphic to $\bQ[x]/P(x)$ where $P(x)=x^4-x^3+2x^2+x+1=0$. The coefficient ring of $f$ has a bais $\{\beta_1,\beta_2,\beta_3\}$ where
\[
 \beta_1=x,
 \qquad
 \beta_2=\frac{1+x^3}{2},
 \qquad
 \beta_3=\frac{-1-2x+2x^2-x^3}{2}.
\]
Then we have
\begin{equation}\label{coefficients}
 a_2(f)=-2-\beta_2,\qquad
 a_3(f)=1+x+\beta_3,\qquad
 \epsilon(3)=\beta_3
\end{equation}
which will be used later.

For every finite place $\lambda$ of $K_f$, Deligne's construction gives a continuous representation
\[
 \rho_{f,\lambda}:G_{\bQ}:=\Gal(\overline{\bQ}/\bQ)\to\GL_2(K_{f,\lambda}).
\]
If $l$ is the residue characteristic of $\lambda$, this
representation is unramified at every prime $q\nmid43\cdot l$ and $\det(\rho_{f,\lambda})=\epsilon\chi_{\cyc}$, where $\chi_{\cyc}$ is the $l$-adic cyclotomic character.

Let $\pi_f$ denote the cuspidal automorphic representation attached to $f$. We record the local type of $f$ at $43$.
\begin{lem}\label{43}
  The local representation $\pi_{f,43}$ is a principal series. If $\lambda\nmid 43$, then $\rho_{f,\lambda}|_{I_{43}}\simeq 1\oplus\epsilon|_{I_{43}}$.
\end{lem}
\begin{proof}
    Note that the conductor of $\pi_{f,43}$ is just $1$, which is the exponent of $43$ in the level of $f$. This excludes the supercuspidal case. If $\pi_{f,43}\cong \chi_{43}\otimes {\rm St}_2$, i.e. it is special,  $\chi_{43}$ must be unramified. Then the central character of $f$ will be $\chi_{43}^2$ which is also unramified, contradicting that the central character $\epsilon|\cdot|_{43}^{-1}$ is ramified at $43$. So $\pi_{f,43}$ is a principal series. And one of the characters attached to $\pi_{f,43}$ is unramified. As $\det(\rho_{f,\lambda})=\epsilon\chi_{\cyc}$, we then have $\rho_{f,\lambda}|_{I_{43}}\simeq 1\oplus\epsilon|_{I_{43}}$.
\end{proof}

For the applications to follow, we use the prime $\lambda=(31,x-16)$ of $K_f$.  As $P(16)\equiv0\pmod {31}$ and $P'(16)\equiv26\not\equiv0\pmod {31}$, we see $16$ is a simple root of $P(x)=0$ modulo $31$ and Hensel's lemma lifts it uniquely to a root in $\bZ_{31}$. Hence we get $K_{f,\lambda}\simeq\bQ_{31}$ and $\calO_{K_{f,\lambda}}\simeq\bZ_{31}$.

Put $x=16$ and then \eqref{coefficients}  gives
\[
 a_2(f)\equiv11,\qquad a_3(f)\equiv22,\qquad \epsilon(3)\equiv5\pmod {31}.
\]

Note that $\rho_{f,\lambda}|_{G_{\bQ_{31}}}$ is non-ordinary as $a_{31}(f)=0$ (cf. \cite{LMFDB43}).

\subsection{The large image at $31$}
Write  $\rho_f:G_\bQ\to\GL_2(\bZ_{31})$ and $\overline\rho_f:G_\bQ\to\GL_2(\bF_{31})$ for a lattice and its reduction attached to $\rho_{f,\lambda}$, respectively. In this section, we will prove a result of large image about $\rho_f,\overline\rho_f$.

The approach is to investigate the orders of elements in $\Ima(\overline\rho_f)$ and use Dickson's classification. Note that if $g\in\GL_2(\bF_{31})$ is semisimple with distinct eigenvalues $\alpha,\beta\in\bF_{31^2}$, the order of the image of $g$ in ${\rm PGL}_2(\bF_{31})$ is the multiplicative order of the eigenvalue ratio $\alpha/\beta$. 

\begin{prop}\label{large-image}
The image $\rho_{f}(G_{\bQ})$ contains $\SL_2(\bZ_{31})$. In particular, $\overline\rho_f(G_\bQ)\supseteq\SL_2(\bF_{31})$.
\end{prop}

\begin{proof}
We first determine the orders of the images of $\overline\rho_f(\Frob_2), \overline\rho_f(\Frob_3)$ in ${\rm PSL}_2(\bF_{31})$, which will be used later. Recall that the characteristic polynomial of $\Frob_q$ for a prime $q\nmid43\cdot31$ is $X^2-a_q(f)X+\epsilon(q)q$. We already know $a_2(f)\equiv11\pmod \lambda$ and $a_3(f)\equiv22 \pmod\lambda$.  Note that $\epsilon(3)=\beta_3\equiv5\pmod\lambda$. Moreover $2\equiv3^{27}\pmod {43}$. Since $\epsilon$ has order $3$,
this gives $\epsilon(2)=\epsilon(3)^{27}=1$.

At $q=2$, the characteristic polynomial is then $X^2-11X+2=(X-27)(X-15)\in\bF_{31}[X]$. The eigenvalue ratio is $27\cdot15^{-1}=8\in\bF_{31}^{\times}$. As $8$ has order $5$ in $\bF_{31}^{\times}$, we see the projective order of $\overline\rho_f(\Frob_2)$ is exactly $5$.

At $q=3$, the determinant is
$3\epsilon(3)\equiv15 \pmod \lambda$, so the characteristic
polynomial is $X^2-22X+15\in\bF_{31}[X]$. Its discriminant is $22^2-4\cdot15\equiv21\pmod {31}$ but $21$ is a nonsquare in $\bF_{31}$.  Choose $\delta\in\bF_{31^2}$ with $\delta^2=21$.  The eigenvalue ratio is then
\[
 r=\frac{22+\delta}{22-\delta}=11+9\delta.
\]
A direct calculation gives $r^{16}=-1$. Thus the projective order of $\overline\rho_f(\Frob_3)$ is exactly $32$.

Now let $\overline H=\overline\rho_f(G_\bQ)\subseteq\GL_2(\bF_{31})$ and $\bP\overline H \subseteq\PGL_2(\bF_{31})$ be the residual image and its projectivization respectively.  We apply Dickson's classification (cf. \cite[Theorem 3.13]{Lombardo15}). If $\bP\overline H$ does not contain $\PSL_2(\bF_{31})$, then $\bP\overline H$ must  be contained in one of the following
\begin{enumerate}
\item a Borel subgroup,
\item the normalizer of a split or nonsplit Cartan subgroup,
\item an exceptional subgroup whose projective image is isomorphic to either $A_4,S_4$, or $A_5$.
\end{enumerate}
In $\PGL_2(\bF_{31})$, the orders of the above three subgroups are
\[
 | B|=31(31-1)=930,
 \qquad
 | N(C_{\rm sp})|=2(31-1)=60,
 \qquad
 | N(C_{\rm ns})|=2(31+1)=64.
\]
The element $\overline\rho_f(\Frob_3)$ of projective order $32$ excludes the Borel subgroup, the split Cartan normalizer and all exceptional groups.  It leaves only the nonsplit Cartan normalizer, but that group has order $64$ and hence contains no element of order $5$.  But the projective order of $\overline\rho_f(\Frob_2)$ is exactly $5$.  It follows that $\bP\overline H\supseteq{\rm PSL}_2(\bF_{31})$. We even have $\bP\overline H={\rm PGL}_2(\bF_{31})$.  Indeed, ${\rm PSL}_2(\bF_{31})$ has index $2$ in ${\rm PGL}_2(\bF_{31})$, so it suffices to show $\bP\overline H$ strictly contains ${\rm PGL}_2(\bF_{31})$. The projective order of $\overline\rho_f(\Frob_3)$ is exactly $32$, yet every nontrivial element of ${\PSL}_2(\bF_{31})$ is either unipotent of order $31$, split semisimple of order dividing $(31-1)/2=15$, or nonsplit semisimple of order dividing $(31+1)/2=16$.

We next move to the linear group.  For any $s,t\in\SL_2(\bF_{31})$, we can  choose elements of $\overline H$ of the form $c_s s$ and $c_t t$, with $c_s,c_t\in\bF_{31}^{\times}$. As $[c_ss,c_tt]=[s,t]\in\overline H$ and the group $\SL_2(\bF_{31})$ is perfect, we then have $\overline H\supseteq\SL_2(\bF_{31})$.

Finally let $H=\rho_f(G_\bQ)$ which is a closed subgroup of
$\GL_2(\bZ_{31})$.  Serre's lifting lemma says that for $p\geq5$, a closed subgroup of $\GL_2(\bZ_p)$ whose reduction contains $\SL_2(\bF_p)$ contains $\SL_2(\bZ_p)$(see \cite[Lemma 3.15]{Lombardo15} or \cite[Chapter IV, \S3.4, Lemma 3]{Serre98}).  So we are done.
\end{proof}

\subsection{The target newform}

Our next goal is to construct for every congruence depth $r\geq1$, a second newform $g_r\neq f$ whose $31$-adic Galois representation is congruent to that of $f$ modulo $31^r$.  We also want that its local representation at $31$ should remain crystalline and non-ordinary, the inertia representation at $43$ should be unchanged and the difference between $f$ and $g_r$ should be concentrated at a controlled nonempty set of new Steinberg primes. Camporino--Pacetti's work \cite{CP18} is designed precisely for this purpose.

We now give the precise specialization of Camporino--Pacetti's construction.  Under the fixed embedding $K_f\hookrightarrow K_{f,\lambda}\simeq\bQ_{31}$, we have $\epsilon(3)\equiv5\pmod{31}$. Since $\epsilon(3)$ has order $3$, it equals the Teichmüller lift $[5]$. Set
\[
\zeta:=\epsilon(3)=[5],
\qquad
b:=\zeta^2=[25].
\]
Then both $\zeta$ and $b$ have order $3$ and $b^2=\zeta^4=\zeta$.
\begin{prop}\label{CP-targets}
For every integer $r\geq1$, there exist a nonempty finite set $Q_r$ of rational primes and a newform
\[
 g_r\in S_2(\Gamma_0(43\prod_{q\in Q_r}q), \epsilon),
\]
a degree one place $\lambda_r\mid31$ of the number field of $g_r$ and a $\bZ_{31}$-lattice $\rho_{g_r}$ in its
$\lambda_r$-adic representation, satisfying the following properties.
\begin{enumerate}
\item There is a congruence of $G_\bQ$-representations $\rho_{g_r}\equiv\rho_f\mod {31}^r$.

\item $\det(\rho_{g_r})=\epsilon\chi_{\cyc}$.

\item The representation $\rho_{g_r}[1/31]$ is crystalline with Hodge--Tate weights $\{0,1\}$ at $31$ and it is nonordinary.

\item Restricted to the inertia subgroup $I_{43}$, we have $\rho_{g_r}|_{I_{43}}\simeq \rho_f|_{I_{43}}$.

\item For every $q\in Q_r$, the source $\rho_f$ is unramified at $q$, yet the target $\rho_{g_r}$ is Steinberg at $q$ and
\begin{equation}\label{negative-auxiliary}
 q\equiv4\pmod {31^r},\qquad \epsilon(q)=\zeta,
\end{equation}
and we also have
\begin{equation}\label{auxiliary-frobenius}
 \rho_f(\Frob_q)\sim
 \begin{pmatrix}u_qq&0\\0&u_q\end{pmatrix}
 \pmod {31^r}
\end{equation}
for a unit $u_q\in(\bZ/31^r\bZ)^{\times}$.  Hence $u_q^2\equiv\zeta\pmod {31^r}$ and $u_q\equiv\pm b\pmod {31^r}$.

\end{enumerate}

\end{prop}
\begin{proof}
We will apply \cite[Theorem B]{CP18} and refine only its choice of auxiliary primes so as to impose conditions \eqref{negative-auxiliary} and keep the conditions in \cite[Section 5]{CP18} at the same time.\\

\noindent\emph{Step 1: verification of the hypotheses.}
Put $\rho_{f,r}=\rho_f\bmod31^r$ and $P_0=\{31,43\}$. The hypotheses of \cite[Theorem B]{CP18} hold with
$p=31$, $k=2$ and $N=43$. In fact, one has $31>\max\{2,3\}$, $31\nmid43$, $K_f=\bQ(\sqrt{-3},\sqrt5)$ unramified at $31$,
$K_{f,\lambda}=\bQ_{31}$ and $\rho_{f,r}$ unramified at $2$.  Proposition \ref{large-image} gives $\SL_2(\bF_{31})\subseteq\overline\rho_f(G_\bQ)$.
Finally, since $a_{31}(f)=0$, by \cite[\S6]{Edixhoven92}, we have
\[
\left(\left.\overline\rho_f\right|_{I_{31}}^{\rm ss}\right)
 \otimes_{\bF_{31}}\overline{\bF}_{31}
 \simeq
 \omega_2\oplus\omega_2^{31}.
\]

Since the conjugation by any lift of $\Frob_{31}$ interchanges the two distinct inertia characters $\omega_2 $ and  $\omega_2^{31} $, neither inertia eigenspace is $G_{\bQ_{31}} $-stable. Hence  $\overline\rho_f|_{G_{\bQ_{31}}} $ is absolutely irreducible. In particular, it cannot be twist-equivalent to either the trivial two-dimensional representation or the indecomposable unramified representation excluded in \cite[Theorem A]{CP18}.\\

\noindent\emph{Step 2: local conditions.}
Camporino--Pacetti fix the determinant throughout their deformation problem \cite[p.1614]{CP18}.  Since $ \det\overline\rho_f=\overline\epsilon\,\overline\chi_{\cyc}$ and the order-three character $\epsilon$ is the Teichm\"uller lift of $\overline\epsilon$, we take the fixed determinant to be $\epsilon\chi_{\cyc}$.

At $43$, we choose the local class containing $\rho_f|_{G_{\bQ_{43}}}$. All its members have the same inertial restriction.  At $31$, we choose the non-ordinary crystalline class containing $\rho_f|_{G_{\bQ_{31}}}$.  The construction in \cite[\S4.1]{CP18} preserves crystallinity and Hodge--Tate weights $\{0,1\}$. The constructed lift will also be non-ordinary.\\

\noindent\emph{Step 3: nice primes.}
For an unramified representation modulo $31^r$ to admit a lift being Steinberg at a prime $q\neq 31$, the ratio of its eigenvalues of $\Frob_q$ must be $q$ modulo $31^r$.  We impose this condition by a mod $p^r$ refinement of the Chebotarev datum used in \cite[\S5.1]{CP18} (in fact, this mod $31^r$ version is also needed and used in \cite{CP18}).

Put
\[
 J_r=\bQ(\ad^0(\rho_{f,r})),\qquad
 K_r^0=J_r\bQ(\mu_{31}),\qquad
 K_r=J_r\bQ(\mu_{31^r}), \qquad  K=\bQ(\ad^0(\overline\rho_f),\mu_{31}).
\]

Note that by Proposition \ref{large-image}, the projective image of $\rho_{f,r}$ contains
\[
 d=\begin{pmatrix}2&0\\0&2^{-1}\end{pmatrix},
\]
whose eigenvalue ratio is $4$ modulo $31^r$. As in the proof of \cite[Proposition 5.1]{CP18},  it can be paired with $4\in(\bZ/31^r\bZ)^\times\cong\Gal(\bQ(\mu_{31^r})/\bQ)$ to get an element in $\Gal(K_r/\bQ)\subseteq \Gal(J_r/\bQ)\times \Gal(\bQ(\mu_{31^r})/\bQ)$.  To see this, let $E_r=J_r\cap\bQ(\mu_{31^r})$. Since  $E_r\subseteq \bQ(\mu_{31^r}) $, the extension  $E_r/\bQ $ is
abelian. Thus the restriction map
\[
\Gal(J_r/\bQ)\twoheadrightarrow{\Gal}(E_r/\bQ)
\]
is an abelian quotient. The group  ${\rm PSL}_2(\bZ/31^r\bZ) $ is perfect, so it lies in the kernel of this map.   So $\Gal(E_r/\bQ)$ is a quotient of $\Gal(J_r/\bQ)/\PSL_2(\bZ/31^r\bZ)$.

On the other hand, the determinant map gives an isomorphism
\[
{\PGL}_2(\bZ/31^r\bZ)/
{\PSL}_2(\bZ/31^r\bZ)
\cong
(\bZ/31^r\bZ)^\times/
((\bZ/31^r\bZ)^\times)^2,
\]
and the group on the right has order  $2 $.  As $\Gal(J_r/\bQ)\subseteq {\rm PGL}_2(\bZ/31^r\bZ)$, we have $[E_r:\bQ]\leq2$.

Thus $E_r$ has degree at most two.  If  $E_r $ is quadratic, the element $4=2^2\in(\bZ/31^r\bZ)^\times\cong\Gal(\bQ(\mu_{31^r})/\bQ)$ then acts trivially on  $E_r $. As $d$  lies in the perfect subgroup ${\PSL}_2(\bZ/31^r\bZ) $, it also acts trivially on this $E_r$.

Hence we get an element $c_r\in\Gal(K_r/\bQ)$ with projective component $d\in \Gal(J_r/\bQ)$ and cyclotomic component $4\in \Gal(\bQ(\mu_{31^r})/\bQ)$. Then a prime $q$ whose Frobenius conjugacy class in $\Gal(K_r/\bQ)$ is the class of $c_r$ satisfies $q\equiv4\pmod {31^r}$.  Thus it is a nice prime as defined in \cite[Section 5]{CP18}, in particular $q\not\equiv\pm1\pmod {31}$ and \eqref{auxiliary-frobenius} holds for some unit $u_q$.

Taking determinant gives $u_q^2q\equiv\det\rho_f(\Frob_q)=\epsilon(q)q\pmod {31^r}$. Hence $u_q^2\equiv\epsilon(q)\pmod {31^r}$. If $\epsilon(q)=\zeta$, then Hensel's lemma gives $u_q\equiv\pm b\pmod {31^r}$.\\

\noindent\emph{Step 4: Chebotarev conditions.}
We now impose $\epsilon(q)=\zeta$ without losing the
cohomological conditions of \cite[\S5.2]{CP18}.  Let
$K_\epsilon/\bQ$ be the cyclic cubic field cut out by
$\epsilon$.  We claim that $K_\epsilon\cap K_r=\bQ$.

Note that $\Gal(K_r/K)$ is a $31$-group. The restriction map identifies $\Gal(K/\bQ(\mu_{31}))\simeq \Gal(J_1/E_1)$. As $E_1/\bQ$ is abelian, the perfect subgroup $\PSL_2(\bF_{31})\subseteq\Gal(J_1/\bQ)$ acts trivially on $E_1$. Hence $\Gal(J_1/E_1)$ contains $\PSL_2(\bF_{31})$, with quotient of order at most $2$. So every cubic quotient of $\Gal(K_r/\bQ)$ factors through
$\Gal(\bQ(\mu_{31})/\bQ)$.  Since $\bQ(\mu_{31})$ is totally ramified at $31$ and $K_\epsilon$ is unramified at $31$, we must have $K_\epsilon\cap K_r=\bQ$.

Let $\gamma_\epsilon$ be the element of $\Gal(K_\epsilon/\bQ)$ on which $\epsilon$ has value $\zeta$.  Chebotarev's density theorem then gives a prime $q_0\notin P_0=\{31,43\}$ whose Frobenius conjugacy class is the class of $(c_r,\gamma_\epsilon)\in\Gal(K_rK_\epsilon/\bQ)=\Gal(K_r/\bQ)\times \Gal(K_\epsilon/\bQ)$.  Hence
\[
 q_0\equiv4\pmod {31^r},\qquad
 \epsilon(q_0)=\zeta.
\]
Prescribe the Steinberg condition at $q_0$ as in \cite[Section 5]{CP18} and put
\[
 P=P_0\cup\{q_0\}.
\]
This ensures a nonempty set with new Steinberg primes.

For this $P$, let $\mathscr F_P$ be the compositum of the fields $L_{f_i}$ and $M_{g_j}$ attached in \cite[\S5.2]{CP18} to the global and dual cohomology classes $f_i\in H^1(G_{\bQ,P},\ad^0\overline\rho_f)$,$g_j\in H^1(G_{\bQ,P},(\ad^0\overline\rho_f)^\vee)$.  By \cite[Lemma 5.6]{CP18}, $\Gal(\mathscr F_P/K)$ is a direct sum of copies of the nontrivial irreducible $G_{\bQ}$-modules
\[
 \Gal(L_{f_i}/K)\cong\ad^0\overline\rho_f
 \quad\text{and}\quad
 \Gal(M_{g_j}/K)\cong (\ad^0\overline\rho_f)^\vee.
\]

We claim that $K_r\cap\mathscr F_P=K$. This indeed follows from that $\Gal(\mathscr F_P/K)$ has no nonzero quotient on which the conjugation action (via any lift) of $\Gal(K/\bQ)$  is trivial. Indeed, such a nonzero quotient will give a nonzero $\Gal(K/\bQ)$-equivariant homomorphism $\ad^0\overline\rho_f\to\bF_{31}$ (or $(\ad^0\overline\rho_f)^\vee\to\bF_{31}$) with trivial action on $\bF_{31}$. This is equivalent to $(\ad^0\overline\rho_f)^{\Gal(K/\bQ)}\neq 0$ due to $\ad^0\overline\rho_f\cong (\ad^0\overline\rho_f)^\vee$. Let $X\in (\ad^0\overline\rho_f)^{\Gal(K/\bQ)}$. Then $X$ is of trace zero and commutes with $\overline\rho(\Gal(K/\bQ))$. Since $\overline\rho(\Gal(K/\bQ))$ contains
 $\SL_2(\bF_{31}) $, we must have
$X=aI_2$ for some $a\in\bF_{31}$.  The trace-zero condition then gives $a=0$.  This proves $(\ad^0\overline\rho_f)^{\Gal(K/\bQ)}=0$.

Now \cite[Lemma 5.9]{CP18} gives $K_r^0\cap\mathscr F_P=K$. Put  $E=K_r\cap\mathscr F_P $.  Since $\mathscr F_P\cap K_r^0=K$, one has $E\cap K_r^0=K$. Hence  $E $ and  $K_r^0 $ are linearly disjoint over  $K $ and $\Gal(E/K)\cong \Gal(EK_r^0/K_r^0)$ is a quotient of the group $\Gal(K_r/K_r^0)$, which is central in $\Gal(K_r/\bQ)$. Thus
$\Gal(K/\bQ)$ acts (via any lift) trivially on $\Gal(E/K)$. Our previous discussion then shows $\Gal(E/K)=1$ and hence $K_r\cap\mathscr F_P=K$.

The extension $K_r\mathscr F_P/K$ is a $31$-extension, whereas $KK_\epsilon/K$ has degree $3$.  Therefore
$K_\epsilon\cap K_r\mathscr F_P=\bQ$.

The restriction of $c_r$ to $K$ is the nice prime class used by Camporino--Pacetti.  Now every Frobenius class prescribed in \cite[Proposition 5.7, Corollary 5.8 and Proposition 5.11]{CP18} can be combined with $c_r$ and $\gamma_\epsilon$. Chebotarev's density theorem  therefore supplies all required auxiliary primes with
\[
 q\equiv4\pmod {31^r},\qquad \epsilon(q)=\zeta.
\]

Let $Q_r$ be $\{q_0\}$ together with the auxiliary primes chosen in the induction, which is particularly non-empty. Since the eigenvalue ratio of every $\Frob_q$ with $q\in Q_r$ is $4\ne1$ modulo $31$ and $\epsilon|_{I_{43}}$ remains nontrivial modulo $31$, the local representations at these primes are already nontrivial modulo $31$. Thus the exceptional trivial
local case of \cite[\S6]{CP18} does not occur.\\

\noindent\emph{Step 5: the target newform.}
\cite[Theorem B]{CP18} now gives a modular lift $\rho_r^{\rm{new}}:G_\bQ\to\GL_2(\bZ_{31})$,
congruent to $\rho_f$ modulo $31^r$, with the prescribed local conditions. It has conductor exponent $0$ at $31$, exponent $1$ at $43$ and exponent $1$ at each $q\in Q_r$, where it is ramified Steinberg and there is no other ramification.  It therefore corresponds to a newform $g_r$ of exact level $ 43\prod_{q\in Q_r}q$. The fixed determinant identifies its nebentype with $\epsilon$. The number field of $g_r$ is contained in $\bQ_{31}$, which gives a degree $1$ place $\lambda_r\mid 31$.  Finally, $q_0\in Q_r$ is a new ramified prime, so $g_r\neq f$.

\end{proof}

\subsection{Final application}

 For a Camporino--Pacetti target $g_{2N}$, we write
\[
 S_f^{(30)}=\End(\Sym^{30}\rho_f)\otimes\epsilon^{-1},\ \quad S_{g,N}^{(30)}=
 \Sym^{30}\rho_{g_{2N}}\otimes(\Sym^{30}\rho_f)^\vee
 \otimes\epsilon^{-1}.
\]
The role of the twist $\epsilon^{-1}$ will be clear in the proofs of Proposition \ref{source-vanish} and Theorem \ref{vanishing transfer}. As $\epsilon|_{G_{\bQ_p}}$ is unramified, it is crystalline and corresponds to a line bundle on $\bZ_p^\syn$ by \cite[Remark 5.25]{terentiuk2026prismatic}.

Both $\rho_{f}|_{G_{\bQ_{31}}}, \rho_{g_{2N}}|_{G_{\bQ_{31}}}$ are Barsotti--Tate representations, with attached reflexive $F$-gauges $M_f, M_{g_{2N}}$ on $\bZ_p^{\syn}$. Hence the set of the Hodge--Tate weights of the above two representations restricted to $G_{\bQ_{31}}$ is $\{-30,-29,\cdots, 0,1,\cdots,30\}$. Thus this example genuinely violates the small ramification--weight inequality $eh<p-1$ for $p=31$.

The congruence $\rho_{g_{2N}}/31^{2N}\simeq \rho_f/31^{2N}$ induces  $S_{g,N}^{(30)}/31^{2N}\simeq S_f^{(30)}/31^{2N}$ and $M_f/p^{2N}\simeq M_{g_{2N}}/p^{2N}$. Applying the tensorial construction to the relevant reflexive $F$-gauges, we get a congruence of the corresponding reflexive $F$-gauges whose \'etale realization is the restriction of $S_{g,N}^{(30)}/31^{2N}\simeq S_f^{(30)}/31^{2N}$.

For any $m\geq 1$, put $F_m=\bQ(\mu_{31^m})$, which is a CM field and will be used when we apply \cite{AHTW26}. Let $a_m^{(30)}$ be the depth threshold in Theorem \ref{main} for the local field $\bQ_{31}(\mu_{31^m})$.  Let $c_m^{(30)}$ be the corresponding exponent for the composite injection in Theorem \ref{cohernel boundedness}. Consider the Selmer complex as in
\eqref{selmer complex}
\[
 C_{f,m}^{(30)}=C_f(F_m,S_f^{(30)},\Sigma_m)
\]
where $\Sigma_m$ consists of the archimedean places and the places above $31$ and $43$ and set
\[
 b_m^{(30)}=\min\{b\geq0:31^bH^2(C_{f,m}^{(30)})_{\tor}=0\}
\]
which only depends on $f$ and $m$.

We first consider Galois cohomology groups of $S_f^{(30)}$.
\begin{prop}\label{source-vanish}
For every $m\geq 1$, one has
\begin{align}
 H_f^1(F_m,S_f^{(30)}[1/31])&=0,\label{source-rational-vanishing}\\
 H^0(F_m,S_f^{(30)}/31)&=0.\label{source H^0}
\end{align}
Moreover $b_m^{(30)}$ is finite and for all $r\geq 1$, we have
\begin{equation}\label{source finite bound}
 31^{b_m^{(30)}}H_f^1(F_m,S_f^{(30)}/31^r)=0.
\end{equation}
\end{prop}

\begin{proof}
We will apply \cite[Corollary 5.3]{AHTW26}.  Let $K_\epsilon/\bQ$ be the cyclic cubic field cut out by
$\epsilon$, which is contained in $\bQ(\mu_{43})$ and is totally real. Put $L=F_mK_\epsilon$. Then the fields $F_m$ and $K_\epsilon$ are linearly disjoint.  Hence $[L:F_m]=3$ and
$L$ is CM.  Since $\epsilon|_{G_L}=1$, we have
\begin{equation}\label{source adjoint over L}
 S_f^{(30)}[1/31]|_{G_L}\simeq\End(\Sym^{30}\rho_f[1/31])|_{G_L}.
\end{equation}

Note that the form $f$ is non-CM (cf. \cite{LMFDB43}). Let $\pi_f$ denote the cuspidal automorphic representation attached to $f$. Newton--Thorne \cite[Theorem A]{NT21} therefore supplies the regular algebraic cuspidal lift $\Sym^{30}\pi_f$ on $\GL_{31}/\bQ$.  Its solvable base change to $L$ remains cuspidal.  Indeed, $L/\bQ$ is abelian, so $G_L$ contains the closed commutator subgroup of $G_\bQ$.  Proposition \ref{large-image} and the perfectness of $\SL_2(\bZ_{31})$ therefore show that $\rho_f(G_L)$ contains $\SL_2(\bZ_{31})$ and $\Sym^{30}\rho_f[1/31]|_{G_L}$ is absolutely irreducible. The solvable base change criterion \cite[Proposition 5.2.2]{ACampo24} then gives a cuspidal base change to $\GL_{31}(\bA_L)$. This base change is regular because the Hodge--Tate weights of $\Sym^{30}\rho_f[1/31]$ are the distinct integers $0,1,\ldots,30$.

Now we consider the unitary normalization $\pi_f^{\rm u}:=\pi_f\otimes |\det|^{1/2}$, whose central character is $\epsilon$, and $\Pi_{f,30}^{\rm u}:=\Sym^{30}\pi_f^{\rm u}=\Sym^{30}\pi_f\otimes |\det|^{15}$. Note that
\[
(\Pi^u_{f,30})^\vee\simeq\Pi^u_{f,30}\otimes\epsilon^{-30}.
\]
 Since $\epsilon^3=1$, one has $\epsilon^{30}=1$ and hence
$\Pi^u_{f,30}$ is self-dual.  Let $\Pi_L:={\rm{BC}}_{L/\bQ}(\Pi^u_{h,30})$ denote its base change to $L$. Then the invariance of a base change from $\bQ$ under the complex conjugation gives
\[
 \Pi_L^c=\Pi_L=\Pi_L^\vee.
\]
Thus $\Pi_L$ is conjugate self-dual.  Its associated $31$-adic representation is $\Sym^{30}\rho_f[1/31]|_{G_L}$, up to the 
normalization, which does not change its adjoint
representation.  These are precisely the regularity and
polarizability hypotheses used in \cite[Corollary 5.3]{AHTW26}.

It remains to verify the enormous image hypothesis.  Put
$M=L(\mu_{31^\infty})$.  Since $M/\bQ$ is abelian, the Galois group $G_M$ contains the closed commutator subgroup of $G_\bQ$.  Proposition \ref{large-image} and the perfectness of $\SL_2(\bZ_{31})$ therefore give $\SL_2(\bZ_{31})\subseteq\rho_f(G_M)$. Conversely, note that both $\epsilon$ and $\chi_{\cyc}$ are trivial on $G_M$. Since $\det\rho_f=\epsilon\chi_{\cyc}$, this proves
\[
 \rho_f(G_M)=\SL_2(\bZ_{31}).
\]

Now put $H:=\Sym^{30}(\rho_f)(G_M)=\Sym^{30}(\SL_2(\bZ_{31}))$. Enlarge the coefficient field to a finite extension $E/\bQ_{31}$ over which the characteristic polynomial of every element of $H$ splits.  By $ \rho_f(G_M)=\SL_2(\bZ_{31})$, the connected Zariski closure of $H$ in $\GL_{31,E}$ is the image of $\SL_{2,E}$ under $\Sym^{30}$ and it acts absolutely irreducibly on $E^{31}$.  It also contains a regular semisimple element: if $a\in1+31\bZ_{31}$ has infinite order, then
\[
 h=\Sym^{30}({\rm{diag}}(a,a^{-1}))
\]
has the pairwise distinct eigenvalues $a^{30},a^{28},\ldots,a^{-30}$.  Hence \cite[Lemma 2.28(2)]{NewtonThorne23} shows that $H$ is enormous in the sense of \cite[Definition 2.23]{NewtonThorne23}.  Therefore \cite[Corollary 5.3]{AHTW26} gives
\begin{equation}\label{adjoint vanishing L}
 H_f^1(L,\End(\Sym^{30}\rho_f[1/31]))=0.
\end{equation}
Set $V=S_f^{(30)}[1/31]$ and let $x\in H_f^1(F_m,V)$.  For every finite place $v$ of $F_m$ and every place $w\mid v$ of $L$, we have a commutative diagram

\[
\begin{array}{ccc}
H^1(F_m,V)
&\xrightarrow{\ {\rm res}_{L/F_m}\ }&
H^1(L,V)
\\[4pt]
{\scriptstyle{\rm loc}_v}\downarrow
&&
\downarrow{\scriptstyle{\rm loc}_w}
\\[4pt]
H^1(F_{m,v},V)
&\xrightarrow{\ {\rm res}_{L_w/F_{m,v}}\ }&
H^1(L_w,V).
\end{array}
\]

If $v\nmid31$, the restriction sends an unramified class
to an unramified class because $I_w\subseteq I_v$.  If $v\mid31$, the restriction preserves the crystalline Bloch--Kato condition by the functoriality of the map
\[
 H^1(F_{m,v},V)\to H^1(F_{m,v},V\otimes_{\bQ_{31}}B_{\cris})
\]
under the inclusion $G_{L_w}\subseteq G_{F_{m,v}}$.  Hence ${\rm{res}}_{L/F_m}(x)\in H_f^1(L,V)$.

The isomorphism \eqref{source adjoint over L} identifies this group with $H_f^1(L,\End(\Sym^{30}\rho_f[1/31]))$, which vanishes by \eqref{adjoint vanishing L}.  Thus ${\rm res}_{L/F_m}(x)=0$.  Note that restriction and corestriction satisfy ${\rm{cor}}_{L/F_m}\circ{\rm{res}}_{L/F_m}=[L:F_m]=3$. It follows that $3x=0$ and therefore $x=0$.  This proves \eqref{source-rational-vanishing}.

For the mod $31$ assertion, note that $\Sym^{30}(\bF_{31}^2)$ is the absolutely irreducible Steinberg representation of $\SL_2(\bF_{31})$.  A $G_{F_m}$-invariant vector in $S_f^{(30)}/31=\End(\Sym^{30}\rho_f/31)\otimes\overline\epsilon^{-1}$, restricted to $G_L$, is therefore a scalar.  This
scalar line is additonally acted on by $\overline\epsilon^{-1}$, which is nontrivial as a character of $G_{F_m}$.  Hence any $G_{F_m}$-invariant vector or the corresponding scalar must be zero, proving \eqref{source H^0}.  Applying Proposition \ref{source torsion exponent} to $C_{f,m}^{(30)}$ then proves the finiteness of $b_m^{(30)}$ and
\eqref{source finite bound}.
\end{proof}

Now we use the congruence to transfer the arithmetic information from $S_f^{(30)}$ to $S_{g,N}^{(30)}$.
\begin{thm}\label{vanishing transfer}
Suppose 
\begin{equation}\label{depth}
 N\geq a_m^{(30)},
 \qquad
 N>b_m^{(30)}+c_m^{(30)}.
\end{equation}
Then we have 
\begin{align*}
 31^{b_m^{(30)}+c_m^{(30)}}
 H_f^1(F_m,S_{g,N}^{(30)}/31^N)&=0,\\
 H_f^1(F_m,S_{g,N}^{(30)})&=0,\\
 H_f^1(F_m,S_{g,N}^{(30)}[1/31])&=0.
\end{align*}

\end{thm}

\begin{proof}

Let
\[
\alpha_{2N}:S_{g,N}^{(30)}/31^{2N}\xrightarrow{\sim}S_f^{(30)}/31^{2N}
\]
be the congruence provided by Proposition \ref{CP-targets} and let $\alpha_N$ be its reduction modulo $31^N$. Let $\frakp_m$ be the unique prime over $31$. Since $N\ge a_m^{(30)}$, Theorem \ref{cohernel boundedness} and Theorem \ref{main} provide the inclusion
\[
31^{c_m^{(30)}}\alpha_{N,*}H_f^1(F_{m,\frakp_m},S_{g,N}^{(30)}/31^N)\subseteq H_f^1(F_{m,\frakp_m},S_f^{(30)}/31^N)
\]

Now let $w\mid43$. By Proposition \ref{CP-targets}, the restrictions of $\rho_f$ and $\rho_{g_{2N}}$ to the inertia subgroup are isomorphic and factor through the same cyclic quotient of order $3$. Thus by Proposition \ref{l}
\[
\alpha_{N,*}H_f^1(F_{m,w},S_{g,N}^{(30)}/31^N)=H_f^1(F_{m,w},S_f^{(30)}/31^N).
\]

At every finite place not lying above $31, 43$, or $Q_{2N}$, both representations are unramified. Hence the local Bloch--Kato Selmer groups are the same.

At a new prime $q\in Q_{2N}$, let $w\mid q$ in $F_m$. We claim $H^1(F_{m,w},S_f^{(30)}/31^N)=H^1(F_{m,w},S_{g,N}^{(30)}/31^N)=0$. We will use Lemma \ref{unramified H^1}. In fact, put $f_w=[k_w:\bF_q]$ where $k_w$ is the residue field. Then $f_w=ord_{(\bZ/31^m\bZ)^\times}(q)$.  By Equation \ref{negative-auxiliary}, we see $q\equiv4\pmod {31}$ has order $5$. As the kernel of $(\bZ/31^m\bZ)^\times\to\bF_{31}^\times$ is a $31$-group, one has $ f_w=5\cdot31^t$ for some $0\leq t\leq m-1$. Note that $\Frob_w=\Frob_q^{f_w}$ and by the Equation \ref{auxiliary-frobenius}, the action of $\Frob_q$ on $\End(\Sym^{30}\rho_f/31)$ is diagonalizable and its eigenvalues are powers $q^k$ for $-30\leq k\leq30$. As $(q^k)^{f_w}=(q^{f_w})^k\equiv 1\mod 31$, we see that $\Frob_w$ acts trivially on $\End(\Sym^{30}\rho_f/31)$. Since $\epsilon^{-1}(\Frob_w)=\epsilon^{-1}(\Frob_q)^{f_w}=\zeta^{-f_w}=\zeta^{-5\cdot 31^t}=\zeta^{-2}=\zeta\neq 1$, we see $\Frob_w$ acts on $S_f^{(30)}/31$ through the scalar $\zeta$. Put $q_w=\#k_w$ and we have $q_w\equiv1\pmod {31}$. Both $\Frob_w-1$ and $\Frob_w-q_w$ are automorphisms on  $S_f^{(30)}/31^N$.  Lemma \ref{unramified H^1} therefore gives
\[
H^1(F_{m,w},S_f^{(30)}/31^r)=H^1(F_{m,w},S_{g,N}^{(30)}/31^N)=0.
\]

Putting all together, we get an inclusion
\[
31^{c_m^{(30)}}\alpha_{N,*}H_f^1(F_{m},S_{g,N}^{(30)}/31^N)\subseteq H_f^1(F_{m},S_f^{(30)}/31^N).
\]

Proposition \ref{source-vanish} provides $p^{b_m^{(30)}}H_f^1(F_{m},S_f^{(30)}/31^N)=0$ (we actually enlarged $\Sigma_m$ by adding places above $Q_{2N}$, but the global Bloch--Kato Selmer group does not change). This proves $31^{b_m^{(30)}+c_m^{(30)}}H_f^1(F_m,S_{g,N}^{(30)}/31^N)=0$.  Equation \ref{source H^0} then gives $H^0(F_m,S_{g,N}^{(30)}/31)=0$.  Finally by Lemma
\ref{target-vanishing}, we get $H_f^1(F_m,S_{g,N}^{(30)})=0$ and $H_f^1(F_m,S_{g,N}^{(30)}[1/31])=0$.
\end{proof}

We next investigate the analytic input needed for the rank part of the Bloch--Kato conjecture.  For a $31$-adic geometric Galois representation $V$ of $\Gal(\overline\bQ/\bQ)$, we write $V^\vee=\Hom_{\bQ_{31}}(V,\bQ_{31})$ and let $L(V,s)$ be its $L$-function, which is the product of all local Euler factors at finite places. Recall that the local Euler factors are defined as follows. For $l\ne31$, let $I_l\subset G_{\bQ_l}$ be the inertia subgroup and let $\Phi_l$ denote the arithmetic Frobenius.  The geometric Frobenius $\Phi_l^{-1}$ acts on $V^{I_l}$ and we set
\begin{equation}\label{L-func}
 L_{l}(V,s)=\det(1-\Phi_l^{-1}l^{-s}\mid V^{I_l})^{-1}.
\end{equation}
At $31$, we use Fontaine's functor $D_{\cris}$ to define the local factor
\[
 L_{31}(V,s)=\det(1-\varphi_{31}31^{-s}\mid D_{\cris}(V))^{-1},
\]
where $\varphi_{31}$ is the crystalline Frobenius (see
\cite[\S4]{BK90}).  Fix $\iota_{31}:\overline{\bQ}_{31}\simeq\bC$. Let  $E/\bQ_{31} $ be a finite extension containing the values of all characters of  $\Delta_m=\Gal(F_m/\bQ) $ and put
\[
\widehat{\Delta}_m=\Hom(\Delta_m,E^\times).
\]
For $\eta\in\widehat\Delta_m$, put
\begin{equation}\label{W30}
 W_{N,\eta}^{(30)}=(S_{g,N}^{(30)}[1/31]\otimes_{\bQ_{31}}E)\otimes\eta^{-1}.
\end{equation}
For $h=f,g_{2N}$, let $\Pi_{h,30}^{\rm u}:=\Sym^{30}\pi_h\otimes |\det|^{15}$ be the unitary normalization of the Newton--Thorne lift $\Sym^{30}\pi_h$.

\begin{prop}\label{local factor}
Let $\Sigma_N=\{31,43\}\cup Q_{2N}$, where $Q_{2N}$ is the set of auxiliary primes obtained in Proposition \ref{CP-targets}. The subscript $\Sigma_N$ below means that the Euler factors at these primes are omitted. For every $\eta\in\widehat\Delta_m$, we have
\begin{equation}\label{2-L}
 L_{\Sigma_N}((W_{N,\eta}^{(30)})^\vee(1),s)=L_{\Sigma_N}(s+1,\Pi_{g_{2N},30}^{\rm u}\times(\Pi_{f,30}^{\rm u}\otimes(\epsilon\eta)^{-1})).
\end{equation}
Moreover, the omitted local Euler factors are finite and nonzero at $s=0$.  In particular, the full  $L$-function satisfies
\[
 L((W_{N,\eta}^{(30)})^\vee(1),0)\neq0.
\]
\end{prop}

\begin{proof}
Dualizing \eqref{W30}, using $(\Sym^dV)^\vee\simeq\Sym^dV\otimes(\det V)^{-d}$ and applying the Tate twist gives
\begin{equation}\label{w-dual}
 (W_{N,\eta}^{(30)})^\vee(1)\simeq \Sym^{30}\rho_{g_{2N}}[1/31]\otimes\Sym^{30}\rho_f[1/31]\otimes\det(\rho_f[1/31])^{-29}\otimes\eta.
\end{equation}
The equality $\epsilon^3=1$ gives
\[
\det(\rho_f[1/31])^{-29}=\epsilon^{-29}\chi_{\cyc}^{-29}=\epsilon\chi_{\cyc}^{-29}.
\]

At a prime $l\notin\Sigma_N$, let $\alpha_{h,l},\beta_{h,l}$ be the roots of
\[
 X^2-a_l(h)X+\epsilon(l)l=0\qquad(h=f,g_{2N}),
\]
and put
\[
 A_{h,i,l}=\alpha_{h,l}^{30-i}\beta_{h,l}^{i}
 \qquad(0\leq i\leq30).
\]
Since $\alpha_{h,l}\beta_{h,l}=\epsilon(l)l$, formula \eqref{L-func} applied to \eqref{w-dual}, gives
\begin{equation}\label{w-local factor}
L_{l}((W_{N,\eta}^{(30)})^\vee(1),s)=\prod_{i,j=0}^{30}(1-(\epsilon\eta)^{-1}(l)A_{g,i,l}A_{f,j,l}l^{-s-31})^{-1}.
 \end{equation}

The Satake parameters of $\Pi_{h,30}^{\rm u}$ at $l$ are
$A_{h,i,l}l^{-15}$.  Thus \eqref{w-local factor} is precisely the local factor at $s+1$ on the right side of \eqref{2-L}, proving the imprimitive identity.

Note that both symmetric power lifts are cuspidal by
\cite[Theorem A]{NT21}.   By \cite[Appendix, Corollary (i) and (ii)]{MoeglinWaldspurger89}, the Rankin--Selberg  $L $-function is holomorphic at  $s+1=1 $ unless
\[
 \Pi_{f,30}^{\rm u}\otimes(\epsilon\eta)^{-1}
 \simeq(\Pi_{g_{2N},30}^{\rm u})^\vee.
\]
We claim that the above isomorphism cannot hold in our case. In fact, for any $q\in Q_{2N}$, the local component of the left hand side has Weil--Deligne monodromy $N=0$, since
$f, \epsilon, \eta$ are all unramified at $q$.
On the other hand, by the construction of $Q_{2N}$, we know  $\pi_{g_{2N},q}\simeq \rm{St}_2\otimes \chi_q$ for a character $\chi_q$. As the conductor of $\pi_{g_{2N},q}$ is $1$, i.e. the exponent of $q$ in the level of $g_{2N}$ is $1$, the character $\chi_q$ is unramified. In particular, the corresponding Weil--Deligne representation is $\chi_q\otimes\rm{Sp}_2$ which has $N\neq 0$.  Then the Weil--Deligne representation of
 $\Pi_{g_{2N},30,q}^{\rm u} $ is $\Sym^{30}(\chi_q\otimes\rm{Sp}_2)\simeq chi_q^{30}\otimes\rm{Sp}_{31}$, whose monodromy is nonzero.   
Now by \cite[Theorem 5.2]{Sha81}, we have 
\[
 L_{\Sigma_N}(1,\Pi_{g_{2N},30}^{\rm u}\times(\Pi_{f,30}^{\rm u}\otimes(\epsilon\eta)^{-1}))\neq0.
\]

It remains to investigate the factors in $\Sigma_N$.  Put
$X_\eta=(W_{N,\eta}^{(30)})^\vee(1)$.  At  $31 $, note that  $X_\eta\simeq X_1\otimes_E\eta $, with $X_1 $ crystalline. Let $L=\bQ_{31}(\mu_{31^m})$. Since  $F_m/\bQ $ is totally ramified at  $31 $, we have $\Gal(L/\bQ_{31})=\Delta_m $ and  $L_0=\bQ_{31} $. Moreover,  $\eta|_{G_L}=1 $. We naturally have
\[
D_{\cris,L}(X_1)\simeq D_{\cris,\bQ_{31}}(X_1),
\]
on which  $\Delta_m $ acts trivially. By definition, we get
\[
D_{\cris,\bQ_{31}}(X_\eta)=D_{\cris,L}(X_\eta)^{\Delta_m}\simeq(D_{\cris,\bQ_{31}}(X_1)\otimes_E \eta)^{\Delta_m}.
\]
The last invariant space is zero. Thus for any $\eta\neq 1$, we have $D_{\cris}(X_\eta)=0 $ and thus $L_{31}(X_\eta,s)=1$. If $\eta=1$, each Frobenius eigenvalue of $X_1$ has complex absolute value $31^{-1}$. Hence $L_{31}(X_1,0)\neq 0$.

For a prime  $l\ne31 $, write ${\rm{WD}}_l(X_\eta)^{\rm{F\text{-}ss}}=(r_{l,\eta},N_{l,\eta})$. Then another equivalent way to define the local Euler factor is
\[
 L_l(X_\eta,s)= \det( 1-\Phi_l^{-1}l^{-s}\mid(\Ker N_{l,\eta})^{I_l})^{-1}.
\]
At  $43 $, the local representation attached to  $f $ is a  principal series by Lemma \ref{43} and hence has zero monodromy.   Since $\rho_{g_{2N}}[1/31]|_{I_{43}}\simeq \rho_f[1/31]|_{I_{43}}$, the  ${\rm{WD}}_{43}(\rho_{g_{2N}}[1/31])^{\rm{F\text{-}ss}}$ also has zero monodromy.  Thus
 $N_{43,\eta}=0 $.  Note that $X_{\eta}$ is pure of weight $30\cdot 1+30\cdot 1-29\cdot 2=2$. So every eigenvalue  of the geometric Frobenius on $r_{43,\eta}^{I_{43}} $ has complex absolute value $43^{-1}$. In particular, $L_{43}(X_\eta,0)\neq 0$.

Now let  $q\in Q_{2N} $.  Note that $\rho_f[1/31], \epsilon,\eta$ are all unramified at $q$. The Steinberg condition on $\rho_{g_{2N}}[1/31]$ means that
\[
\rm{WD}_q(\rho_{g_{2N}}[1/31])^{\rm{F\text{-}ss}}\simeq (\chi_q\otimes{\rm Sp}_2, N_g)
\]for an unramified character $\chi_q$ and $N_g\neq 0$. Equivalently, the corresponding local Galois representation fits into a nonsplit exact sequence
\[
0\to E(1)\otimes\chi_q\to \rho_{g_{2N}}[1/31]\to E\otimes\chi_q
\to0
\]
and $\chi_q^2=\epsilon$. We have $\Ker N_g=E(1)\otimes\chi_q$. Note that $\chi_{\cyc}(\Phi_q^{-1})=q^{-1}$ and the eigenvalue of the Frobenius on $\chi_q$ has complex absolute value $1$. Hence the eigenvalue of the geometric Frobenius on $\Ker N_g$ has complex absolute value $q^{-1}$. It follows that every eigenvalue of the geometric Frobenius on $(\Ker N_{q,\eta})^{I_q}$ has absolute value $q^{-30}\cdot q^{-15}\cdot q^{29}\cdot1=q^{-16}$. In particular, $L_q(X_{\eta},0)\neq 0$. Finally, we get $L((W_{N,\eta}^{(30)})^\vee(1),0)\neq0$.

\end{proof}
Now we can state our final result concerning the rank part of the Bloch--Kato conjecture for $W_{N,\eta}^{(30)}$.

\begin{cor}\label{BK}
Assume the equalities \eqref{depth}.  For every $\eta\in\widehat\Delta_m$, we have
\[
 H_f^1(\bQ,W_{N,\eta}^{(30)})=H^0(\bQ,W_{N,\eta}^{(30)})={\rm ord}_{s=0}L((W_{N,\eta}^{(30)})^\vee(1),s)=0.
\]

As a result, the rank equality predicted by Bloch--Kato holds:
\[
{\rm ord}_{s=0}L((W_{N,\eta}^{(30)})^\vee(1),s)=\dim_EH_f^1(\bQ,W_{N,\eta}^{(30)})-\dim_E H^0(\bQ,W_{N,\eta}^{(30)})=0.\]

\end{cor}

\begin{proof}
As the Bloch--Kato local conditions are compatible with Shapiro's lemma, we have
\[
 H_f^1(F_m,S_{g,N}^{(30)}[1/31])\otimes_{\bQ_{31}}E\simeq H_f^1(\bQ,{\rm Ind}_{G_{F_m}}^{G_\bQ}((S_{g,N}^{(30)}[1/31]\otimes_{\bQ_{31}}E)|_{G_{F_m}})).\]

Then the projection formula gives
\[
H_f^1(\bQ,{\rm Ind}_{G_{F_m}}^{G_\bQ}((S_{g,N}^{(30)}[1/31]\otimes_{\bQ_{31}}E)|_{G_{F_m}}))\simeq H_f^1(\bQ, (S_{g,N}^{(30)}[1/31]\otimes_{\bQ_{31}}E)\otimes_EE[\Delta_m]).
\]
As $\Delta_m$ is finite abelian and $E$ contains the values of all characters of $\Delta$, we then get
\[
 H_f^1(F_m,S_{g,N}^{(30)}[1/31])\otimes_{\bQ_{31}}E\simeq \bigoplus_{\eta\in\widehat\Delta_m} H_f^1(\bQ,W_{N,\eta}^{(30)}).
\]
The left side vanishes by Theorem \ref{vanishing transfer}, proving $H_f^1(\bQ,W_{N,\eta}^{(30)})=0$. Moreover, we have
\begin{equation}\label{W-Hom}
 H^0(\bQ,W_{N,\eta}^{(30)})\simeq\Hom_{E[G_\bQ]}(\Sym^{30}\rho_f[1/31]\otimes\epsilon\eta,\Sym^{30}\rho_{g_{2N}}[1/31]).
\end{equation}
Both representations on the right are irreducible.  In fact, by Proposition \ref{large-image}, the image of $\bar \rho_f=\bar\rho_{g_{2N}}$ contains $\SL_2(\bF_{31})$. As $\Sym^{30}(\bF_{31}^2)$ is absolutely irreducible as an $\SL_2(\bF_{31})$-module, we see $\Sym^{30}(\rho_f/31)=\Sym^{30}(\rho_{g_{2N}}/31)$ is irreducible, which implies both $\Sym^{30}\rho_f[1/31]$ and $\Sym^{30}\rho_{g_{2N}}[1/31]$ are irreducible.

Hence a nonzero element in \eqref{W-Hom} would be an
isomorphism.  This is impossible at any prime $q\in Q_{2N}$. In fact, the source is unramified at $q$, thus has zero Weil--Deligne monodromy, whereas the target is Steinberg at $q$, thus has nonzero Weil--Deligne monodromy. This proves $H^0(\bQ,W_{N,\eta}^{(30)})=0$.

The identity ${\rm ord}_{s=0}L((W_{N,\eta}^{(30)})^\vee(1),s)=0$ directly follows from  Proposition \ref{local factor}. Then all the three terms in the rank part of Bloch--Kato conjecture are $0$. So we are done.

\end{proof}

\addcontentsline{toc}{section}{References}

\bibliographystyle{alpha}
\bibliography{bibliography}

@article{terentiuk2026prismatic,
  title={Prismatic {F}-gauges and {F}ontaine--{L}affaille modules},
  author={Terentiuk, Gleb and Vologodsky, Vadim and Xu, Yujie},
  journal={Compositio Mathematica},
  volume={162},
  number={2},
  pages={277--326},
  year={2026},
  publisher={London Mathematical Society}
}

@article{gardner2024algebraicity,
  title={An algebraicity conjecture of {D}rinfeld and the moduli of $p$-divisible groups},
  author={Gardner, Zachary and Madapusi, Keerthi},
  journal={arXiv preprint arXiv:2412.10226},
  year={2024}
}

@article{MoeglinWaldspurger89,
  author  = {M{\oe}glin, Colette and Waldspurger, Jean-Loup},
  title   = {Le spectre r{\'e}siduel de {$\operatorname{GL}(n)$}},
  journal = {Ann. Sci. \'{E}cole Norm. Sup. (4)},
  volume  = {22},
  number  = {4},
  pages   = {605--674},
  year    = {1989},
  doi     = {10.24033/asens.1595},

}

@article{NewtonThorne23,
  author  = {Newton, James and Thorne, Jack A.},
  title   = {Adjoint {Selmer} groups of automorphic {Galois}
             representations of unitary type},
  journal = {J. Eur. Math. Soc.},
  volume  = {25},
  number  = {5},
  pages   = {1919--1967},
  year    = {2023},
  doi     = {10.4171/JEMS/1228},

}

@article{ACampo24,
  author  = {A'Campo, Lambert},
  title   = {Rigidity of automorphic {Galois} representations over {CM} fields},
  journal = {Int. Math. Res. Not. IMRN},
  volume  = {2024},
  number  = {6},
  pages   = {4541--4623},
  year    = {2024},
  doi     = {10.1093/imrn/rnad087},

}

@misc{AHTW26,
  author       = {A'Campo, Lambert and Hevesi, Bence and Thorne, Jack A. and Whitmore, Dmitri},
  title        = {Adjoint {Bloch--Kato} {Selmer} groups of regular algebraic automorphic {Galois} representations},
  howpublished = {arXiv:2607.11769v1},
  year         = {2026},

}

@incollection{BK90,
  author    = {Bloch, Spencer and Kato, Kazuya},
  title     = {{$L$-functions and {Tamagawa} numbers of motives}},
  booktitle = {The Grothendieck Festschrift, Vol. I},
  series    = {Progress in Mathematics},
  volume    = {86},
  pages     = {333--400},
  publisher = {Birkh{\"a}user},
  address   = {Boston},
  year      = {1990},

}

@article{CP18,
  author  = {Camporino, Maximiliano and Pacetti, Ariel},
  title   = {Congruences between modular forms modulo prime powers},
  journal = {Rev. Mat. Iberoam.},
  volume  = {34},
  number  = {4},
  pages   = {1609--1643},
  year    = {2018},
  doi     = {10.4171/RMI/1037},

}

@article{Edixhoven92,
  author  = {Edixhoven, Bas},
  title   = {The weight in {Serre}'s conjectures on modular forms},
  journal = {Invent. Math.},
  volume  = {109},
  number  = {3},
  pages   = {563--594},
  year    = {1992}
}

@misc{LMFDB43,
  author = {The {LMFDB Collaboration}},
  title  = {Newform orbit $43.2.\mathrm{c}.\mathrm{b}$},
  note   = {\url{https://www.lmfdb.org/ModularForm/GL2/Q/holomorphic/43/2/c/b/}}
}

@article{Lombardo15,
  author  = {Lombardo, Davide},
  title   = {Bounds for {Serre}'s open image theorem for elliptic curves over number fields},
  journal = {Algebra Number Theory},
  volume  = {9},
  number  = {10},
  pages   = {2347--2395},
  year    = {2015}
}

@article{Nekovar06,
  author  = {Nekov{\'a}\v{r}, Jan},
  title   = {{Selmer} complexes},
  journal = {Ast{\'e}risque},
  volume  = {310},
  year    = {2006},
 
}

@article{NT21,
  author  = {Newton, James and Thorne, Jack A.},
  title   = {Symmetric power functoriality for holomorphic modular forms, {II}},
  journal = {Publ. Math. Inst. Hautes \'{E}tudes Sci.},
  volume  = {134},
  pages   = {117--152},
  year    = {2021},
  doi     = {10.1007/s10240-021-00126-4},
 
}

@article{Raynaud74,
  author  = {Raynaud, Michel},
  title   = {Sch{\'e}mas en groupes de type $(p,\ldots,p)$},
  journal = {Bull. Soc. Math. France},
  volume  = {102},
  pages   = {241--280},
  year    = {1974},
  
}

@book{Serre98,
  author    = {Serre, Jean-Pierre},
  title     = {Abelian $\ell$-adic Representations and Elliptic Curves},
  series    = {Research Notes in Mathematics},
  volume    = {7},
  publisher = {A K Peters},
  address   = {Wellesley, MA},
  year      = {1998}
}

@article{Sha81,
  author  = {Shahidi, Freydoon},
  title   = {On certain {$L$}-functions},
  journal = {Amer. J. Math.},
  volume  = {103},
  number  = {2},
  pages   = {297--355},
  year    = {1981},
  doi     = {10.2307/2374219},

}

@article{antieau2024k,
  title={On the {K}-theory of $\mathbb{Z}/p^n$},
  author={Antieau, Benjamin and Krause, Achim and Nikolaus, Thomas},
  journal={arXiv preprint arXiv:2405.04329},
  year={2024}
}

@article{pentland2025syntomification,
  title={Syntomification and crystalline local systems},
  author={Pentland, Dylan},
  journal={arXiv preprint arXiv:2510.16961},
  year={2025}
}

@article{gao2024integral,
  title={Integral filtered {S}en theory and applications},
  author={Gao, Hui and Liu, Tong},
  journal={arXiv preprint arXiv:2411.11084},
  year={2024}
}

@article{bhatt22,
  title={Prismatic {F}-gauges},
  author={Bhatt, Bhargav and Lurie, Jacob},
  journal={Lecture notes, available at https://www. math. ias. edu/\~{} bhatt/teaching/mat549f22/lectures. pdf},
  year={2022}
}

@article{wu2021galois,
  title={Galois representations, ({$\varphi$}, {$\Gamma$})-modules and prismatic {$F$}-crystals},
  author={Wu, Zhiyou},
  journal={Doc. Math},
  volume={26},
  pages={1771--1798},
  year={2021}
}

@article{gabber2025bounding,
  title={Bounding crystalline torsion from \'etale torsion},
  author={Gabber, Ofer and Li, Shizhang},
  journal={arXiv preprint arXiv:2506.13543},
  year={2025}
}

@article{Coupek26,
  title={Exactness property of {B}reuil--{K}isin functors and {B}loch--{K}ato {S}elmer groups},
  author={{\v C}oupek, Pavel and Gazaki, Evangelia and  Marmora, Adriano},
  journal={arXiv:2603.26035},
  year={2026}
}

@article{DZ,
  title={A prismatic {H}err complex for {B}loch--{K}ato {S}elmer groups},
  author={Du, Heng and Zhao, Luming},
  journal={arXiv:2611.37653},
  year={2026}
}

@article{gao2023breuil,
  title={{B}reuil--{K}isin modules and integral $p$-adic {H}odge theory},
  author={Gao, Hui},
  journal={Journal of the European Mathematical Society},
  volume={25},
  number={10},
  pages={3979--4032},
  year={2023},
  publisher={EMS Press}
}

@article{Abhinandan,
  title={Crystalline part of the {G}alois cohomology of crystalline representations},
  author={Abhinandan},
  journal={to appear in Bulletin de la Société Mathématique de France},
  year={}
}

@article{iovita15,
  title={On the continuity of the finite {B}loch--{K}ato cohomology},
  author={Iovita, Adrian and Marmora, Adriano},
  journal={Rendiconti del Seminario Matematico della Universit{\`a} di Padova},
  volume={134},
  pages={239--272},
  year={2015}
}

@article{bs22,
  title={Prisms and prismatic cohomology},
  author={Bhatt, Bhargav and Scholze, Peter},
  journal={Annals of Mathematics},
  volume={196},
  number={3},
  pages={1135--1275},
  year={2022},
  publisher={Department of Mathematics, Princeton University Princeton, New Jersey, USA}
}

@article {BS23,
    AUTHOR = {Bhatt, Bhargav and Scholze, Peter},
     TITLE = {Prismatic {$F$}-crystals and crystalline {G}alois
              representations},
   JOURNAL = {Camb. J. Math.},
  FJOURNAL = {Cambridge Journal of Mathematics},
    VOLUME = {11},
      YEAR = {2023},
    NUMBER = {2},
     PAGES = {507--562},
      ISSN = {2168-0930,2168-0949},
   MRCLASS = {14F30 (11F80 14G45)},
  MRNUMBER = {4600546},
}
\end{document}